\documentclass[a4paper,11pt,reqno]{amsart}

\usepackage[utf8]{inputenc}
\usepackage[T1]{fontenc}
\usepackage{lmodern}
\usepackage[english]{babel}
\usepackage{microtype}

\usepackage{amsmath,amssymb,amsfonts,amsthm,comment}
\usepackage{mathtools,accents}
\usepackage{mathrsfs}
\usepackage{aliascnt}
\usepackage{braket}
\usepackage{bm}
\usepackage{esint}
\usepackage{amsfonts}
\usepackage{csquotes}

\usepackage[a4paper,twoside,top=0.8in, bottom=0.7in, left=0.7in, right=0.7in]{geometry} 
\usepackage{hyperref}
\usepackage[english,capitalize]{cleveref}

\usepackage{enumerate}
\usepackage{xcolor}
\usepackage{comment}

\usepackage{imakeidx} %Analytical index
\usepackage{xparse} %In order to use ExplSyntax (in older latex versions)
\usepackage{mathtools}
\usepackage{float}
\usepackage{comment}
\usepackage{nicefrac}
\usepackage{graphicx}
\usepackage{enumerate}

\definecolor{lava}{rgb}{0.81, 0.06, 0.13}

\hypersetup{colorlinks=true, linkcolor=lava, citecolor=lava, urlcolor  = lava}

\usepackage[maxbibnames=99,backend=biber, sorting=nyt, doi=false, url=false, isbn=false, style=alphabetic]{biblatex}
\makeatletter
\g@addto@macro\@floatboxreset\centering
\makeatother

\makeatletter
\def\newaliasedtheorem#1[#2]#3{
  \newaliascnt{#1@alt}{#2}
  \newtheorem{#1}[#1@alt]{#3}
  \expandafter\newcommand\csname #1@altname\endcsname{#3}
}
\makeatother

\numberwithin{equation}{section}

\newtheoremstyle{slanted}{\topsep}{\topsep}{\slshape}{}{\bfseries}{.}{.5em}{}

\theoremstyle{plain}
\newtheorem{theorem}{Theorem}[section]
\newaliasedtheorem{proposition}[theorem]{Proposition}
\newaliasedtheorem{question}[theorem]{Question}
\newaliasedtheorem{problem}[theorem]{Problem}
\newaliasedtheorem{lemma}[theorem]{Lemma}
\newaliasedtheorem{corollary}[theorem]{Corollary}
\newaliasedtheorem{counterexample}[theorem]{Counterexample}

\theoremstyle{definition}
\newaliasedtheorem{definition}[theorem]{Definition}
\newaliasedtheorem{example}[theorem]{Example}
\newaliasedtheorem{conjecture}[theorem]{Conjecture}

\theoremstyle{remark}
\newaliasedtheorem{remark}[theorem]{Remark}
\newcommand{\N}{\mathbb{N}}

\newcommand{\setR}{\mathbb{R}}
\newcommand{\R}{\mathbb{R}}
\renewcommand{\S}{\mathbb{S}}

\newcommand{\mm}{\mathfrak m}
\newcommand{\X}{{\rm X}}
\newcommand{\lip}{{\rm lip \,}}
\newcommand{\nchi}{{\raise.3ex\hbox{\(\chi\)}}}
\newcommand{\eps}{\varepsilon}

\let\phi\varphi

\newcommand{\di}{\mathop{}\!\mathrm{d}}

\newcommand{\res}{\mathop{\hbox{\vrule height 7pt width .5pt depth 0pt
\vrule height .5pt width 6pt depth 0pt}}\nolimits}

\newcommand{\Ch}{{\sf Ch}}
\newcommand{\st}{\ensuremath{\ :\ }} % Such that in formulas
\newcommand{\eqdef}{\ensuremath{\vcentcolon=}}

\newcommand{\haus}{\mathcal{H}}

\newcommand{\dist}{\mathsf{d}}

\newcommand{\meas}{\mathfrak{m}}

\newcommand{\diam}{\mathrm{diam}}
\newcommand{\CD}{\mathsf{CD}}
\newcommand{\RCD}{\mathsf{RCD}}

\DeclareMathOperator{\AVR}{AVR}

\DeclareMathOperator{\Ric}{Ric}
\DeclareMathOperator{\Per}{Per}

\newfont{\tmpf}{cmsy10 scaled 2500}

\newcommand{\de}{\ensuremath{\,\mathrm d}} % Il de degli integrali (contiene lo spazio da mettere tra integranda e misura)
\renewcommand{\d}{{\rm d}}

\date{\today}

\begin{document}

\title[Isoperimetry on manifolds with Ricci bounded below]{Isoperimetry on manifolds with Ricci bounded below:\\ overview of recent results and methods}

\author{Marco Pozzetta}
\address{Dipartimento di Matematica e Applicazioni, Universit\`a di Napoli Federico II, Via Cintia, Monte S. Angelo 80126 Napoli, Italy.}
\email{marco.pozzetta@unina.it}

\maketitle

\begin{abstract}
We review recent results on the study of the isoperimetric problem on Riemannian manifolds with Ricci lower bounds.

We focus on the validity of sharp second order differential inequalities satisfied by the isoperimetric profile of possibly noncompact Riemannian manifolds with Ricci lower bounds. We give a self-contained overview of the methods employed for the proof of such result, which exploit modern tools and ideas from nonsmooth geometry. The latter methods are needed for achieving the result even in the smooth setting.

Next, we show applications of the differential inequalities of the isoperimetric profile, providing simplified proofs of: the sharp and rigid isoperimetric inequality on manifolds with nonnegative Ricci and Euclidean volume growth, existence of isoperimetric sets for large volumes on manifolds with nonnegative Ricci and Euclidean volume growth, the classical L\'{e}vy--Gromov isoperimetric inequality.

On the way, we discuss relations of these results and methods with the existing literature, pointing out several open problems.
\end{abstract}

\tableofcontents

\section{Introduction}

Let $(M,g)$ be a possibly noncompact $N$-dimensional complete Riemannian manifold. For $V \in(0,\haus^N(M))$, the classical \emph{isoperimetric problem} aims at studying the minimization
\[
\inf\left\{ P(E) \st E \subset M, \, \haus^N(E)=V \right\},
\]
where $P(E)$ denotes the perimeter of $E$, which, roughly speaking, measures the $(N-1)$-dimensional volume of the boundary of $E$. The previous infimum as a function of the volume $V$ is called \emph{isoperimetric profile} and denoted by $I_M(V)$. A minimizer is called \emph{isoperimetric set}.

\medskip

There is a classical connection between the isoperimetric problem on manifolds and lower bounds on the Ricci curvature, going back to the L\'{e}vy--Gromov isoperimetric inequality at least, see \cite{GromovLevyGromovOriginale} and \autoref{sec:LevyGromov} below. 
In this work, we aim at discussing this connection by reviewing recent results from \cite{AFP21, AntonelliNardulliPozzetta, AntBruFogPoz, AntonelliPasqualettoPozzetta21, APPSa, APPSb} in a self-contained exposition.

A celebrated result coming from the seminal works \cite{BavardPansu86, Gallotast} states that the isoperimetric profile of a compact manifold satisfies a second order \emph{differential inequality} depending on the lower bound on the Ricci curvature and on the dimension. These differential inequalities are classically derived by computing the second variation of the perimeter of an isoperimetric set, hence they rely on \emph{existence} of minimizers, which is possibly false on general noncompact manifolds with Ricci lower bounds \cite{Rit01NonExistence, CaneteRitore, AFP21}. Nonetheless, in the recent \cite{APPSa}, the validity of these differential inequalities is generalized to noncompact manifolds, yielding the next result, on which we shall focus our attention for the rest of the work.

\begin{theorem}[{Sharp differential inequalities of the isoperimetric profile \cite[Theorem 1.1]{APPSa}, cf. \autoref{thm:DifferentialInequalitiesProfile}}]\label{thm:DifferentialInequalitiesProfileINTRO}
Let $N \in \N$ with $N\ge 2$, and let $K\in\R$. Let $(M,g)$ be an $N$-dimensional complete Riemannian manifold with $\Ric\ge K$. Assume that there exists $v_0>0$ such that $\haus^N(B_1(x)) \ge v_0$ for any $x \in M$.
Let $\psi\eqdef I_M^{\frac{N}{N-1}}$. Then $\psi$ solves
\begin{equation*}
    \psi'' \le - \frac{K\, N}{N-1} \psi^{\frac{2-N}{N}},
\end{equation*}
in the sense of distributions on $(0,\haus^N(M))$, equivalently in the viscosity sense\footnote{For the precise definitions we refer to the beginning of \autoref{sec:SharpDiffIneq}.} on $(0,\haus^N(M))$, equivalently $\overline{D}^2 \psi(V)\le - \tfrac{K\, N}{N-1} \psi^{\frac{2-N}{N}}(V)$ for any $V \in (0,\haus^N(M))$, where $\overline{D}^2$ is the upper second derivative, see \eqref{eq:DefD2}.
\end{theorem}

The proof of \autoref{thm:DifferentialInequalitiesProfileINTRO} is based on several preliminary results which exploit tools and methods from \emph{nonsmooth geometry}, especially from the theory of $\RCD$ spaces, which we regard here as a generalization of the concept of Riemannian manifold with Ricci bounded below. Note that the proof of \autoref{thm:DifferentialInequalitiesProfileINTRO} that we shall outline \emph{cannot avoid} the use of such methods at the moment.\\
In order to overcome the possible nonexistence of isoperimetric sets on noncompact manifolds, in \cite{AFP21, AntonelliNardulliPozzetta} after \cite{Lions84I, Nar14, RitRosales04}, it has been set up an approach by \emph{direct method} to identify generalized isoperimetric sets under the sole assumptions of \autoref{thm:DifferentialInequalitiesProfileINTRO}. Such result will be recalled in \autoref{thm:AsymptoticMassDecomposition} and it exploits the natural \emph{precompactness} of sequences of spaces with lower Ricci bounds. In particular, the mentioned generalized isoperimetric sets are contained in (the union of) possibly nonsmooth spaces different from the starting ambient manifold.
Hence, in order to derive the desired differential inequalities, it is not possible to perform the classical argument by second variation of the area, as the same computation is out of reach in the nonsmooth realm at the moment. Instead, by proving topological regularity of isoperimetric sets - \autoref{thm:Regularity} - and existence of a weak notion of barrier on the mean curvature - \autoref{thm:MeanCurvatureBarriers} - it is possible to give a sharp estimate on the second variation of the perimeter, sufficient to deduce \autoref{thm:DifferentialInequalitiesProfileINTRO}. Note that this argument completely avoids any deeper regularity theory about boundaries of isoperimetric sets.

\medskip

To show how powerful \autoref{thm:DifferentialInequalitiesProfileINTRO} is, we shall explicitly exploit the differential inequalities of the profile to provide proofs of:
\begin{enumerate}
\item the sharp and rigid isoperimetric inequality on manifolds with nonnegative Ricci curvature and Euclidean volume growth \cite{AgostinianiFogagnoloMazzieri, BrendleFigo, BaloghKristaly, APPSb, CavallettiManini}, see \autoref{thm:DisugisoperimetricaAVR};

\item a general existence result of isoperimetric sets for large volumes on manifolds with nonnegative Ricci curvature and Euclidean volume growth \cite{AntBruFogPoz}, see \autoref{thm:ABFPexistence};

\item the L\'{e}vy--Gromov isoperimetric inequality \cite{GromovLevyGromovOriginale, Gromovmetric}, see \autoref{thm:LevyGromov}.
\end{enumerate}

Proofs of results (1) and (2) above bring simplifications to previous proofs from \cite{APPSb} and \cite{AntBruFogPoz}. The proof of (3) above is essentially a review of the argument from \cite{Bayle03}, with a simplification of the rigidity part.

\medskip

Results below are stated and proved in the context of $\RCD$ spaces. However, all the necessary preliminaries are briefly recalled in \autoref{sec:Preliminaries} and gradually stated when needed, aiming at a self-contained presentation.

We shall also discuss history and literature related to each result in the course of the note, as well as important open problems on the topic.

\medskip

\noindent\textbf{Organization.} In \autoref{sec:Preliminaries} we collect fundamental definitions and facts on $\RCD$ spaces, convergence, sets of finite perimeter and isoperimetric problem. In \autoref{sec:Main} we review the main tools needed for the proof of the differential inequalities of the profile, namely the asymptotic mass decomposition \autoref{thm:AsymptoticMassDecomposition} and the existence of mean curvature barriers \autoref{thm:MeanCurvatureBarriers}; then we outline the proof of \autoref{thm:DifferentialInequalitiesProfileINTRO}. In \autoref{sec:Applications} we collect the above mentioned applications of \autoref{thm:DifferentialInequalitiesProfileINTRO} to the isoperimetric problem on spaces with nonnegative curvature. Finally, \autoref{sec:Appendix} is devoted to basic facts on concave functions and ODE comparison.

\medskip

\noindent\textbf{Acknowledgements.} I am partially supported by the INdAM - GNAMPA Project 2022 CUP \_ E55F22000270001 ``Isoperimetric problems: variational and geometric aspects''. I would like to warmly thank Gioacchino Antonelli, Elia Bruè, Mattia Fogagnolo, Stefano Nardulli, Enrico Pasqualetto, Daniele Semola and Ivan Yuri Violo for countless inspiring discussions on the isoperimetric problem and related topics. I also thank Valentina Franceschi, Alessandra Pluda and Giorgio Saracco for having organized the very nice workshop ``Anisotropic Isoperimetric Problems \& Related Topics'' and for the opportunity to write this work for the conference proceedings of the event.

\section{Preliminaries}\label{sec:Preliminaries}

In this section we briefly recall the concept of $\RCD$ space as a generalization of the notion of Riemannian manifold with Ricci bounded below, together with basic definitions and facts on sets of finite perimeter and on the isoperimetric problem in this framework. We want to stress how the smooth theory naturally extends to this setting, discussing some examples as well.

\subsection{$\RCD$ spaces, examples and Gromov--Hausdorff convergence}

For the sake of simplicity, we will say that a triple $(X,\dist,\meas)$ is a \emph{metric measure space}, shortly m.m.s., if $(X,\dist)$ is a locally compact separable metric space and $\meas$ is a nonnegative Radon measure on $X$ such that ${\rm spt}\,\meas=X$.

\medskip

The \emph{Cheeger energy} on a metric measure space \((X,\dist,\meas)\) is defined as the \(L^2\)-relaxation of the functional
\(f\mapsto\frac{1}{2}\int\lip^2 f\de\meas\), see \cite{AmbrosioGigliSavare11} after \cite{Cheeger99}, where the \emph{slope} $\lip  f$ of a locally Lipschitz function $f$ is defined by
\[
\lip  f(x) \eqdef \limsup_{y\to x} \frac{|f(y)-f(x)|}{\dist(x,y)},
\]
if $x \in X$ is not isolated, while $\lip f(x)=0$ is $x$ is isolated. Hence, for any function \(f\in L^2(X)\) we define
\[
\Ch(f)\coloneqq\inf\bigg\{\liminf_{i}\,\frac{1}{2}\int\lip^2 f_i\di\meas \st  f_i \in {\rm Lip}_c(X),\,f_i\to f\text{ in }L^2(X)\bigg\}\, .
\]
The \emph{Sobolev space} \(H^{1,2}(X)\) is defined as the finiteness domain \(\{f\in L^2(X)\,:\,\Ch(f)<+\infty\}\) of the Cheeger energy, thus naturally extending the usual definition of smooth manifolds.\\
The restriction of the Cheeger energy to the Sobolev space admits the integral representation \(\Ch(f)=\frac{1}{2}\int|\nabla f|^2\di\meas\),
for a uniquely determined function \(|\nabla f|\in L^2(X)\) that is called the \emph{minimal weak upper gradient} of \(f\in H^{1,2}(X)\).
The linear space \(H^{1,2}(X)\) is a Banach space if endowed with the Sobolev norm
\[
\|f\|_{H^{1,2}(X)}\coloneqq\sqrt{\|f\|_{L^2(X)}^2+2\Ch(f)}=\sqrt{\|f\|_{L^2(X)}^2+\||\nabla f|\|_{L^2(X)}^2},\quad\text{ for every }f\in H^{1,2}(X)\, .
\]
Following \cite{GigliOnTheDiffStructure}, when \(H^{1,2}(X)\) is a Hilbert space (or equivalently \(\Ch\) is a quadratic form) we say that the metric measure
space \((X,\dist,\meas)\) is \emph{infinitesimally Hilbertian}.\\
We further define the mapping \(H^{1,2}(X)\times H^{1,2}(X)\ni(f,g)\mapsto\nabla f\cdot\nabla g\in L^1(X)\) as
\[
\nabla f\cdot\nabla g\coloneqq\frac{|\nabla(f+g)|^2-|\nabla f|^2-|\nabla g|^2}{2},\quad\text{ for every }f,g\in H^{1,2}(X)\, .
\]
We define the \emph{Laplacian} as follows: we define \(D(\Delta)\subset H^{1,2}(\X)\) as the
space of all functions \(f\in H^{1,2}(\X)\) for which there exists
(a uniquely determined) \(\Delta f\in L^2(\mm)\) such that
\[
\int \nabla f\cdot\nabla g\de \mm=-\int g\,\Delta f\de\mm,
\quad\text{ for every }g\in H^{1,2}(\X),
\]
Let us mention that it is possible to derive an effective generalized first order calculus on infinitesimally Hilbertian spaces following \cite{GigliOnTheDiffStructure, Gigli17}, generalizing the usual calculus of the smooth setting.

\medskip

With the above terminology, we can introduce the definition of the so-called $\RCD$ condition. For more on the topic, we refer the interested reader to the survey \cite{AmbrosioSurvey} and references therein. Let us just mention that, after the introduction in the seminal independent works \cite{Sturm1,Sturm2} and \cite{LottVillani} of the \emph{curvature dimension condition} $\CD(K,N)$, encoding in a synthetic way the notion of Ricci curvature bounded from below by $K \in \R$ and dimension bounded above by $N \in [1,+\infty)$, the definition of \emph{Riemannian curvature dimension condition} $\RCD(K,N)$ for a metric measure space was first proposed in \cite{GigliOnTheDiffStructure} and then studied in \cite{Gigli13, ErbarKuwadaSturm15,AmbrosioMondinoSavare15}. See \cite{CavallettiMilmanCD, LiGlobalization} for the equivalence between the $\RCD^*(K,N)$ and the $\RCD(K,N)$ condition. The infinite dimensional counterpart of this notion has been previously investigated in \cite{AmbrosioGigliSavare14, AmbrosioGigliMondinoRajala15}.

\begin{definition}[\(\RCD(K,N)\) space]\label{def:RCD}
Let \((\X,\dist,\meas)\) be a metric measure space. Then
\((\X,\dist,\meas)\) is an \emph{\({\sf RCD}(K,N)\) space}, for some
\(K\in\R\) and \(N\in[1,\infty)\), provided the following conditions hold:
\begin{itemize}
\item There exist \(C>0\) and \(\bar x\in\X\) such that
\(\meas(B_r(\bar x))\leq e^{Cr^2}\) for every \(r>0\).
\item \textsc{Sobolev-to-Lipschitz property.} If \(f\in H^{1,2}(\X)\)
satisfies \(|\nabla f|\in L^\infty(\meas)\), then \(f\) admits a Lipschitz
representative \(\bar f\colon\X\to\R\) such that
\({\rm Lip}(\bar f)=\big\||\nabla f|\big\|_{L^\infty(\meas)}\).
\item \((\X,\dist,\meas)\) is infinitesimally Hilbertian.
\item \textsc{Bochner inequality.} It holds that
\[
\frac{1}{2}\int|\nabla f|^2\Delta g\de\meas\geq
\int g\bigg(\frac{(\Delta f)^2}{N}+
\nabla f\cdot\nabla\Delta f+K|\nabla f|^2\bigg)\de\meas,
\]
for every \(f\in D(\Delta)\) with \(\Delta f\in H^{1,2}(\X)\)
and \(g\in D(\Delta)\cap L^\infty(\meas)\) nonnegative with
\(\Delta g\in L^\infty(\meas)\).
\end{itemize}
An $\RCD(K,N)$ space $(X,\dist,\haus^N)$ is said to be \emph{noncollapsed}\footnote{By the rectifiability results \cite{MondinoNaber, KellMondino, GigliPasqualettoReferenceMeasure, BruePasqualettoSemola20}, it follows that for a noncollapsed $\RCD(K,N)$ space, the number $N$ is necessarily integer. See also \cite{ChCo0, CheegerColdingI} for the first contributions to the study of noncollapsed spaces arising as limits of manifolds, and the recent \cite{BrenaGigliHondaZhu} for the equivalence with the notion of weakly noncollapsedness.}. Instead, we will refer to any $\RCD(K,N)$ space $(X,\dist,\meas)$ endowed with reference measure different from $\haus^N$ as to a \emph{collapsed} space.
\end{definition}

We want to stress the similarities of these spaces with the classical Riemannian manifolds with Ricci bounded below.\\
First, \autoref{def:RCD} is clearly consistent with the smooth setting, i.e., if $(M,g)$ is a complete $N$-dimensional Riemannian manifold with Riemannian distance $\dist$, then $(M,\dist,\haus^N)$ is $\RCD(K,N)$ if and only if $\Ric\ge K$. In fact, observe that the Bochner inequality among the axioms of $\RCD(K,N)$ spaces is nothing but the natural weak integral formulation of the standard Bochner inequality on manifolds, the latter being the inequality derived from the Bochner identity \cite{Petersen2016} estimating from below the terms involving the norm of the Hessian and the Ricci tensor.\\
Concerning the first three axioms in \autoref{def:RCD}, as already commented above, the reader may think that they ensure a well-established first order calculus that extends the classical one \cite{GigliOnTheDiffStructure}. Actually, the definition of $\RCD(K,N)$ space is sufficient to recover a variety of classical result in Riemannian Geometry in this generalized setting. Some of these results are recalled below and in \autoref{sec:GeomAnal}, and allow to perform a powerful Geometric Analysis in this setting. We just mention here that the $\RCD(K,N)$ condition as in \autoref{def:RCD} on a m.m.s. $(X,\dist,\meas)$ implies that $\mm$ is uniformly locally doubling, hence the space is proper, and moreover that $(X,\dist)$ is geodesic, i.e., any two points are joined by a curve of length equal to their distance, in particular the space is path-connected.

\medskip

Obviously, considering problems in Geometric Analysis in the framework of $\RCD$ spaces (as we shall do) is not just for the mathematical quest of the greatest generality. Instead, one benefits of new properties enjoyed by the $\RCD$ class, such as the fundamental \emph{precompactness} with respect to Gromov--Hausdorff topology recalled in the next \autoref{sec:GH}, which will be essential in the study of the isoperimetric problem and in the proof of the main result on the differential properties of the isoperimetric profile, \autoref{thm:DifferentialInequalitiesProfileINTRO}.

\subsubsection{Examples}

Before continuing with preliminaries, we discuss a few examples needed in the sequel. Apart from complete Riemannian manifolds with lower Ricci bounds, prototypical examples of $\RCD(K,N)$ spaces are given by weighted manifolds: if $(M,g)$ is an $m$-dimensional complete Riemannian manifold with geodesic distance $\dist$ and $V\in C^\infty(M)$, then $(M,\dist, e^{-V}\haus^m)$ is $\RCD(K,N)$ for $N\ge m$ if and only if the \emph{generalized $N$-Ricci curvature}
\[
\Ric_N \eqdef \Ric + {\rm Hess}\, V - \frac{\nabla V \otimes \nabla V}{N-m},
\]
is bounded below by $K$, where $V$ is assumed to be constant and the last term is defined to be zero if $N=m$. Moreover, (possibly weighted) complete Riemannian manifolds with (generalized) Ricci lower bounds with convex boundary belong to the $\RCD$ class \cite{HanMnfBdry}, where here convex means that the second fundamental form of the boundary with respect the inner normal is nonnegative.

In the sequel, we shall mostly focus on noncollapsed $\RCD(K,N)$ spaces $(X,\dist,\haus^N)$ of dimension $N\ge2$, referring to extensions to the collapsed case when possible. Also, whenever we write a triple $(X,\dist,\haus^N)$, it is tacitly understood that $N$-dimensional Hausdorff measure is computed with respect to the distance in the triple. Within the class of noncollapsed $\RCD(K,N)$ spaces we also find Euclidean convex bodies, i.e., closures of open convex sets in Euclidean spaces, as well as boundaries of convex sets in $\R^N$ (endowed with corresponding intrinsic distance and $(N-1)$-Hausdorff measure). Convex bodies and their boundaries actually belong to the more restrictive class of Alexandrov spaces with nonnegative curvature, we refer to \cite{BuragoBuragoIvanovBook} for a definition; in fact, any $N$-dimensional Alexandrov space with curvature bounded below by $k\in\R$ endowed with $N$-dimensional Hausdorff measure is in particular $\RCD((N-1)k,N)$, see \cite{PetruninAlexandrovCD, ZhangZhuAlexandrovRCD, GigliKuwadaOhta}.

\medskip

We will be particularly interested in two specific constructions, namely cones and spherical suspensions, that we define here in the noncollapsed setting (see \cite{Ketterer15} for the general case).

\begin{definition}\label{def:ConiSospensioni}
Let $(X,\dist_X)$ be a compact metric space with $\diam(X)\le \pi$. Let $N\in \N$ with $N\ge 2$.
\begin{itemize}
    \item The \emph{(Euclidean metric) cone over $X$} is the metric space $\big(C(X), \dist_C\big)$ where $C(X)\eqdef [0,+\infty)\times X /_{\{0\}\times X}$ and
    \[
    \dist_C((t,x),(s,y)) \eqdef \big(t^2 + s^2 -2ts \cos(\dist_X(x,y)) \big)^{\frac12}.
    \]
    There holds that $(X,\dist_X,\haus^{N-1})$ is $\RCD(N-2,N-1)$ if and only if $(C(X),\dist_C,\haus^N)$ is $\RCD(0,N)$ \cite{Ketterer15}.\\
    Denoting by $\bar{o}$ the point $\{0\}\times X$ in the quotient $C(X)$, any point $o \in C(X)$ such that there exists an isometry $j:C(X)\to C(X)$ such that $j(\bar{o}) = o$ is called a \emph{tip}, or \emph{vertex}, of the cone.

    \item The \emph{spherical suspension over $X$} is the metric space $\big(S(X), \dist_S\big)$ where $S(X)\eqdef [0,\pi]\times X /_{\{0,\pi\}\times X}$ and
    \[
    \dist_S((t,x),(s,y)) \eqdef \cos^{-1}\big(\cos(t)\,\cos(s) + \sin(t)\,\sin(s) \cos(\dist_X(x,y)) \big).
    \]
    There holds that $(X,\dist_X,\haus^{N-1})$ is $\RCD(N-2,N-1)$ if and only if $(S(X),\dist_C,\haus^N)$ is $\RCD(N-1,N)$ \cite{Ketterer15}.\\
    Points $\{0\}\times X$, $\{\pi\}\times X$ in the quotient $S(X)$ are called \emph{poles}.
\end{itemize}
\end{definition}

Observe that closed convex cones contained in the Euclidean space are $\RCD(0,N)$ cones in the sense of \autoref{def:ConiSospensioni}.

If $(X,\dist_X)$ in \autoref{def:ConiSospensioni} is a smooth Riemannian manifold $(X,g_X)$, then the cone over $X$ coincides with the (smooth out of set of tips) manifold given by the warped product
\[
\big( [0,+\infty) \times X, \d t^2+ t^2 \,g_X\big),
\]
while the spherical suspension over $X$ coincides with the (smooth out of set $\{0,\pi\}\times X \subset S(X)$) manifold given by the warped product
\[
\big( [0,\pi] \times X, \d t^2+ \sin^2(t) \, g_X\big).
\]
In particular, $C(X)$ (resp. $S(X)$) is a smooth manifold if and only if $X$ is the standard sphere $\S^{N-1}$, in which case $C(X)=\R^N$ (resp. $S(X)=\S^N$).

We shall regard $\RCD(0,N)$ cones as prototypical examples of noncompact spaces with nonnegative Ricci curvature and large volume growth. Analogously, $\RCD(N-1,N)$ spherical suspensions are seen as reference examples of compact spaces with positive Ricci curvature and diameter equal to $\pi$. In terms of the isoperimetric problem, we will see that this picture shall be confirmed by the fundamental isoperimetric inequalities we will prove in \autoref{thm:DisugisoperimetricaAVR} and \autoref{thm:LevyGromov}. Analogously, cones and spherical suspensions represent the rigidity cases in the generalized Bishop--Gromov \autoref{thm:BishopGromov} and in the generalized Bonnet--Myers \autoref{thm:BonnetMyers} we shall recall below.

\subsubsection{Gromov--Hausdorff convergence}\label{sec:GH}

A fundamental advantage of working in the $\RCD(K,N)$ class is the natural precompactness with respect to pointed measure Gromov--Hausdorff convergence. The following definition is taken from the introductory exposition of \cite{AmbrosioBrueSemola19}, see also \cite{BuragoBuragoIvanovBook, GigliMondinoSavare15}.

\begin{definition}[pGH and pmGH convergence]\label{def:GHconvergence}
A sequence $\{ (X_i, \dist_i, x_i) \}_{i\in \N}$ of pointed metric spaces is said to converge in the \emph{pointed Gromov--Hausdorff topology}, shortly \emph{pGH}, to a pointed metric space $ (Y, \dist_Y, y)$ if there exist a complete separable metric space $(Z, \dist_Z)$ and isometric embeddings
\[
\begin{split}
&\Psi_i:(X_i, \dist_i) \to (Z,\dist_Z), \qquad \forall\, i\in \N\, ,\\
&\Psi:(Y, \dist_Y) \to (Z,\dist_Z)\, ,
\end{split}
\]
such that for any $\eps,R>0$ there is $i_0(\varepsilon,R)\in\mathbb N$ such that
\[
\Psi_i(B_R^{X_i}(x_i)) \subset \left[ \Psi(B_R^Y(y))\right]_\eps,
\qquad
\Psi(B_R^{Y}(y)) \subset \left[ \Psi_i(B_R^{X_i}(x_i))\right]_\eps\, ,
\]
for any $i\ge i_0$, where $[A]_\eps\coloneqq \{ z\in Z \st \dist_Z(z,A)\leq \eps\}$ for any $A \subset Z$. In other words, for any $R>0$, $\Psi_i(B_R^{X_i}(x_i))$ converges to $\Psi(B_R^{Y}(y))$ with respect to Hausdorff distance in $(Z,\dist_Z)$.

Let $\meas_i$ and $\mu$ be measures such that $(X_i,\dist_i,\meas_i,x_i)$ and $(Y,\dist_Y,\mu,y)$ are metric measure spaces. If in addition to the previous requirements we also have $(\Psi_i)_\sharp\mathfrak{m}_i \rightharpoonup \Psi_\sharp \mu$ with respect to duality with continuous bounded functions on $Z$ with bounded support, then the convergence is said to hold in the \emph{pointed measure Gromov--Hausdorff topology}, shortly \emph{pmGH}.
\end{definition}

The previous notions of convergence can be regarded as a generalization of the convergence in Hausdorff distance of compact sets in the Euclidean space. In particular, if $X_i$ in \autoref{def:GHconvergence} is a sequence of smooth manifolds, the pGH (or pmGH) limit of $X_i$ is not a smooth manifold in general.

On the other hand, the $\RCD$ condition is naturally \emph{stable} with respect pmGH convergence, yielding the above mentioned precompactness property that we now state specialized to the noncollapsed case.

\begin{theorem}[Precompactness]\label{thm:Precompactness}
Let $N \in \N$ with $N\ge 2$, let $K\in\R$, and let $v_0>0$. Let $(X_i,\dist_i,\haus^N,x_i)$ be a sequence of $\RCD(K,N)$ spaces such that $\inf_{i} \haus^N(B_1(x_i)) \ge v_0$. Then, up to subsequence, $(X_i,\dist_i,\haus^N,x_i)$ pmGH converges to an $\RCD(K,N)$ space $(X,\dist_X,\haus^N,x)$.
\end{theorem}

In \autoref{thm:Precompactness} and in the sequel, assumptions like $\inf_{x} \haus^N(B_1(x)) \ge v_0>0$ are regarded as \emph{noncollapsing hypotheses}, as they guarantee the pmGH limit to be endowed with Hausdorff measure, in accordance with \autoref{def:RCD}.

\autoref{thm:Precompactness} follows from the fact that: (i) the class of $\RCD(K,N)$ spaces is closed with respect to pmGH convergence \cite{LottVillani, Sturm1, Sturm2, AmbrosioGigliSavare14, GigliOnTheDiffStructure, GigliMondinoSavare15}, (ii) the reference measure on an $\RCD(K,N)$ space is locally uniformly doubling (this follows, for instance, from \autoref{thm:BishopGromov} below) and thus Gromov precompactness applies (see \cite[Sect. 5.A]{Gromovmetric} or \cite{Petersen2016}), and (iii) the stability of noncollapsedness \cite[Theorem 1.2]{DePhilippisGigli18} (see also \cite{Colding97, ChCo1}).

\medskip

In the following, we shall very often consider pointed sequences of the form $(X,\dist,\haus^N,x_i)$ for a fixed $\RCD(K,N)$ space $X$ and for a diverging sequence of points $x_i \in X$, i.e., such that $\limsup_i\dist(o,x_i)=+\infty$ for any $o \in X$. In this case, any pGH limit of such a sequence is called \emph{a limit at infinity of $X$}.

\medskip

A celebrated result concerning a sequence of spaces as in \autoref{thm:Precompactness} is that the volume of balls centered at points $x_i$ converges along the sequence, only requiring convergence in pGH sense.

\begin{theorem}[{Volume convergence, \cite[Theorem 1.2, Theorem 1.3]{DePhilippisGigli18} after \cite{Colding97, CheegerColdingI}}]\label{thm:VolumeConvergence}
Let $N \in \N$ with $N\ge 2$, let $K\in\R$, and let $v_0>0$. Let $(X_i,\dist_i,\haus^N,x_i)$ be a sequence of $\RCD(K,N)$ spaces such that $\inf_{i} \haus^N(B_1(x_i)) \ge v_0$. Assume that $(X_i,\dist_i,x_i)$ converges in pGH sense to a metric space $(X,\dist_X,x)$. Then $(X,\dist_X,\haus^N,x)$ is $\RCD(K,N)$, convergence holds in pmGH sense, and $\lim_i \haus^N(B_R(x_i)) = \haus^N(B_R(x))$ for any $R>0$.
\end{theorem}

Comparing with \autoref{thm:Precompactness}, roughly speaking, the previous \autoref{thm:VolumeConvergence} tells that, assuming pGH-convergence, then pmGH-convergence to the limit space endowed with Hausdorff measure is equivalent to the convergence of the volume of balls.

\subsection{Sets of finite perimeter}\label{sec:SetsFinitePerimeter}

There is a well-established theory of functions of bounded variations on metric measure spaces \cite{Ambrosio02, Miranda03, DiMarinoThesis}, allowing the treatment of sets of finite perimeter in this generalized setting.

Let $(X,\dist,\meas)$ be a metric measure space. The definition of $BV$ function is then given by relaxation by approximation with locally Lipschitz functions, thus extending the classical notion from Euclidean spaces or Riemannian manifolds \cite{AFP00}.

\begin{definition}[$\rm BV$ functions and perimeter on m.m.s.]\label{def:BVperimetro}
Let $(X,\dist,\meas)$ be a metric measure space.  Given $f\in L^1_{\mathrm{loc}}(X,\meas)$ we define
\[
|Df|(A) \eqdef \inf\left\{\liminf_i \int_A \lip f_i \de\meas \st \text{$f_i \in {\rm Lip}_{\mathrm{loc}}(A),\,f_i \to f $ in $L^1_{\mathrm{loc}}(A,\meas)$} \right\}\, ,
\]
for any open set $A\subset X$.
A function \(f\in L^1_{\mathrm{loc}}(X,\meas)\) is of \emph{local bounded variation}, briefly \(f\in{\rm BV}_{\mathrm{loc}}(X)\),
if \(|Df|(A)<+\infty\) for every \(A\subset X\) open bounded.
A function $f \in L^1(X,\meas)$ belongs to the space of \emph{functions of bounded variation} ${\rm BV}(X)={\rm BV}(X,\dist,\meas)$ if $|Df|(X)<+\infty$. 

If $E\subset\X$ is a Borel set and $A\subset X$ is open, we  define the \emph{perimeter $\Per(E,A)$  of $E$ in $A$} by
\[
\Per(E,A) \eqdef \inf\left\{\liminf_i \int_A \lip u_i \de\meas \st \text{$u_i \in {\rm Lip}_{\mathrm{loc}}(A),\,u_i \to \nchi_E $ in $L^1_{\mathrm{loc}}(A,\meas)$} \right\}\, ,
\]
in other words \(\Per(E,A)\coloneqq|D\nchi_E|(A)\).
We say that $E$ has \emph{locally finite perimeter} if $\Per(E,A)<+\infty$ for every open bounded set $A$. We say that $E$ has \emph{finite perimeter} if $\Per(E,X)<+\infty$, and we denote $\Per(E)\eqdef \Per(E,X)$.
\end{definition}

Let us remark that when $f\in{\rm BV}_{\mathrm{loc}}(X,\dist,\meas)$ or $E$ is a set with locally finite perimeter, the set functions $|Df|, \Per(E,\cdot)$ above are restrictions to open sets of Borel measures that we still denote by $|Df|, \Per(E,\cdot)$, see \cite{AmbrosioDiMarino14, Miranda03}.

\medskip

Recalling \autoref{def:GHconvergence}, it is naturally possible to speak of convergence of sets along sequences of converging spaces.

\begin{definition}[$L^1$-strong and $L^1_{\mathrm{loc}}$ convergence]\label{def:L1strong}
Let $\{ (X_i, \dist_i, \mathfrak{m}_i, x_i) \}_{i\in \N}$  be a sequence of pointed metric measure spaces converging in the pmGH sense to a pointed metric measure space $ (Y, \dist_Y, \mu, y)$ and let $(Z,\dist_Z)$ be a realization as in \autoref{def:GHconvergence}.

We say that a sequence of Borel sets $E_i\subset X_i$ such that $\mathfrak{m}_i(E_i) < +\infty$ for any $i \in \N$ converges \emph{in the $L^1$-strong sense} to a Borel set $F\subset Y$ with $\mu(F) < +\infty$ if $\mathfrak{m}_i(E_i) \to \mu(F)$ and $\chi_{E_i}\mathfrak{m}_i \rightharpoonup \chi_F\mu$ with respect to the duality with continuous bounded functions with bounded support on $Z$.

We say that a sequence of Borel sets $E_i\subset X_i$ converges \emph{in the $L^1_{\mathrm{loc}}$-sense} to a Borel set $F\subset Y$ if $E_i\cap B_R(x_i)$ converges to $F\cap B_R(y)$ in $L^1$-strong for every $R>0$.
\end{definition}

Sets of finite perimeter in the $\RCD$ framework enjoy the usual precompactness, approximation, and lower semicontinuity properties with respect to $L^1_{\rm loc}$ convergence.

\begin{remark}[Precompactness and lower semicontinuity of finite perimeter sets along pmGH converging sequences]\label{rem:SemicontPerimeterConverging}
Let $K\in\mathbb R$, $N\geq 1$, and $\{(X_i,\dist_i,\meas_i,x_i)\}_{i\in\mathbb N}$ be a sequence of $\RCD(K,N)$ metric measure spaces converging in the pmGH sense to $(Y,\dist_Y,\mu,y)$. Let $(Z,\dist_Z)$ be a realization of the convergence. Then, the following hold, compare with \cite[Proposition 3.3, Corollary 3.4, Proposition 3.6, Proposition 3.8]{AmbrosioBrueSemola19}, and \cite{AmbrosioHonda17}.
\begin{itemize}
    \item For any sequence of Borel sets $E_i\subset X_i$ with 
    $$
    \sup_{i\in\mathbb N}|D\chi_{E_i}|(B_R(x_i))<+\infty, \qquad \forall\,R>0,
    $$
    there exists a subsequence $i_k$ and a Borel set $F\subset Y$ such that $E_{i_k}\to F$ in $L^1_{\mathrm{loc}}$.
    \item Let $F\subset Y$ be a bounded set of finite perimeter. Then there exist a subsequence $i_k$ and uniformly bounded sets of finite perimeter $E_{i_k}\subset X_{i_k}$ such that $E_{i_k}\to F$ in $L^1$-strong and $|D\chi_{E_{i_k}}|(X_{i_k})\to |D\chi_F|(Y)$ as $k\to+\infty$.
    % \item Let $f_i\in \mathrm{BV}(X_i,\dist_i,\mathfrak{m}_i)$ converge in $L^1$-strong to $f\in L^1(Y,\mu)$. If $\sup_i|Df_i|(X_i)<+\infty$, then $f\in \mathrm{BV}(Y,\dist_Y,\mu)$, and we have
    % \[
    % |Df|(Y)\leq \liminf_{i\to+\infty}|Df_i|(X_i).
    % \]
\end{itemize}
\end{remark}

In the last years, several fine properties of sets of finite perimeter in $\RCD$ spaces have been proved \cite{BPSrectifiabilityJEMS, BPSGaussGreen, BruePasqualettoSemola20}, generalizing the Euclidean theory \cite{AFP00}. Given a Borel set \(E\subset X\) in a noncollapsed \(\RCD(K,N)\) space \((X,\dist,\haus^N)\) and any \(t\in[0,1]\), we denote by \(E^{(t)}\) the set
of \emph{points of density \(t\)} of \(E\), namely
\[
E^{(t)}\coloneqq\bigg\{x\in X\;\bigg|\;\lim_{r\to 0}\frac{\haus^N(E\cap B_r(x))}{\haus^N(B_r(x))}=t\bigg\}\, .
\]
The \emph{essential boundary} of \(E\) is defined as \(\partial^e E\coloneqq X\setminus(E^{(0)}\cup E^{(1)})\). It is also possible to speak of \emph{reduced boundary} \(\mathcal F E\subset\partial^e E\) of a set of locally finite perimeter \(E\subset X\), that is defined as the set of the points of \(X\) where the unique tangent to \(E\), up to isomorphism, is the half-space, see \cite[Definition 4.1]{AmbrosioBrueSemola19} for the precise definition.\\
It was proved in \cite{BPSrectifiabilityJEMS}, after \cite{Ambrosio02,AmbrosioBrueSemola19}, that the perimeter measure has the representation
\begin{equation}\label{eq:RepresentationPerimeter}
\Per(E,\cdot)=\haus^{N-1}|_{\mathcal F E},
\end{equation}
Moreover, according to \cite[Proposition 4.2]{BPSGaussGreen}, 
\begin{equation}\label{eq:DeGFederer}
\mathcal F E=E^{(1/2)}=\partial^e E
\qquad\text{ up to }\mathcal H^{N-1}\text{-null sets}\, ,
\end{equation}
generalizing De Giorgi's and Federer's theorems to the $\RCD$ setting, see \cite[Theorem 3.59, Theorem 3.61]{AFP00}. Let us mention that the representation of the perimeter measure is today well-understood also in the case of collapsed $\RCD$ spaces, see \cite{BPSGaussGreen} and \cite[Section 3]{AntonelliBrenaPasqualettoRankOne}.

\begin{remark}\label{rem:RelativePerimeter}
    Let $(X,|\cdot|, \haus^N)$ be a convex body in $\R^N$. It readily follows from \eqref{eq:RepresentationPerimeter} and \eqref{eq:DeGFederer} that the perimeter of a set $E\subset X$ is automatically the \emph{relative} perimeter of $E$ in the interior of $X$.
\end{remark}

\subsection{Isoperimetric problem, profile and sets}

\begin{definition}[Isoperimetric profile and isoperimetric sets]
Let $(X,\dist,\meas)$ be a metric measure space. We define the \emph{isoperimetric profile}
\[
I_X(V) \eqdef \inf\left\{P(E) \st E\subset X \text{ Borel, } \meas(E)=V \right\}
\]
for any $V\in(0,\meas(X))$. Set also $I_X(0)\eqdef 0$, and $I_X(\meas(X))\eqdef 0$ if $\meas(X)<+\infty$.

A Borel set $E\subset X$ such that $\meas(E) \in (0,\meas(X))$ and $P(E)= I_X(\meas(E))$ is called \emph{isoperimetric set}, or \emph{isoperimetric region}.
\end{definition}

\begin{remark}
    Let $N \in \N$ with $N\ge 2$, and let $K\in\R$. Let $(X,\dist,\haus^N)$ be an $\RCD(K,N)$ space. Assume that there exists $v_0>0$ such that $\haus^N(B_1(x)) \ge v_0$ for any $x \in X$. Then the isoperimetric profile $I_X:[0,\haus^N(X))\to [0,+\infty)$ is a $(1-\tfrac1N)$-H\"{o}lder continuous real valued function by \cite[Theorem 2]{FloresNardulli20}, after \cite[Lemme 6.2]{Gallotast}, \cite[Lemma 6.9]{Milmanconvexity} (see also the related argument in \cite[Lemma 3.4]{Buser82}). In fact, exploiting the main result \autoref{thm:DifferentialInequalitiesProfile}, one can prove interior local Lipschitz regularity, see \autoref{rem:RegularityNew} and \autoref{rem:InfPositivo} below.\\
    On the other hand, it is possible to construct complete Riemannian manifolds with discontinuous isoperimetric profile for any dimension greater or equal to $2$, see \cite{NardulliPansu, PapasogluSwenson}.\\
    If $\haus^N(X)<+\infty$, then $I_X$ is continuous on the whole interval $[0,\haus^N(X)]$ and, since perimeter is invariant with respect to complement, $I_X$ is symmetric around $\haus^N(X)/2$.
\end{remark}

A fundamental question about isoperimetric sets addresses their regularity. While on a Riemannian manifold it makes sense to speak about finer regularity properties of isoperimetric sets, in the nonsmooth setting we can at least investigate their topological regularity.

\begin{theorem}[{Topological regularity, \cite[Theorem 1.4]{AntonelliPasqualettoPozzetta21} \& \cite{AntonelliPasqualettoPozzettaViolo}}]\label{thm:Regularity}
Let $N \in \N$ with $N\ge 2$, and let $K\in\R$. Let $(X,\dist,\haus^N)$ be an $\RCD(K,N)$ space. Assume that there exists $v_0>0$ such that $\haus^N(B_1(x)) \ge v_0$ for any $x \in X$. Let $E\subset X$ be an isoperimetric set.\\
Then $E^{(1)}$ is open and bounded, $\partial^e E = \partial E^{(1)}$, and $E^{(0)}$ is open.
% and $\partial E^{(1)}$ is $(N-1)$-Ahlfors regular, i.e., there exist $C,r_0>0$ such that
%\[
%C^{-1}r^{N-1} \le \haus^{N-1}(\partial E^{(1)} \cap B_r(x)) \le Cr^{N-1},
%\]
%for any $x \in \partial E$ and $r\in(0,r_0)$.
\end{theorem}

Thanks to the recent \cite{AntonelliPasqualettoPozzettaViolo}, the previous result also holds in the generality of possibly collapsed $\RCD(K,N)$ spaces $(X,\dist,\meas)$.

Topological regularity for isoperimetric sets in the Euclidean spaces was firstly proved in \cite{GonzalezMassariTamanini}, and subsequently generalized in \cite{Xia05}. In the proof of \autoref{thm:Regularity}, it plays a crucial role the so-called \emph{deformation property} which allows to increase or decrease the volume of sets of finite perimeter controlling the change of perimeter in terms of the change of volume, see \cite[Theorem 1.1, Theorem 2.35]{AntonelliPasqualettoPozzetta21}. Deformation properties, well-known in the smooth context \cite[Lemma 17.21]{MaggiBook}, have great importance in several arguments, see \cite[VI.2(3)]{AlmgrenBook}, \cite[Lemma 13.5]{MorganBook}, \cite[Lemma 4.5]{GalliRitore}, \cite[Lemma 3.6]{Pozuelo}, \cite{CintiPratelli, PratelliSaracco}.

We mention that a version of \autoref{thm:Regularity} holds for local volume constrained minimizers of quasi-perimeters, that is, functionals given by the sum of the perimeter and of a suitable $L^1$-continuous term. Also, \autoref{thm:Regularity} implies further minimality properties on isoperimetric sets, like $\Lambda$-minimality, see \cite[Theorem 3.24]{AntonelliPasqualettoPozzetta21} and \cite[Chapter 21]{MaggiBook}, and thus density estimates \cite[Proposition 3.27]{AntonelliPasqualettoPozzetta21} (see also \autoref{rem:RegularityNew} below).

We further observe that exploiting uniform $\Lambda$-minimality properties of isoperimetric sets \cite[Corollary 4.17]{APPSa}, one can adapt arguments from the regularity theory for perimeter minimizers developed in \cite{MoS21} to prove that the Hausdorff dimension of the singular set of the boundary of an isoperimetric set, i.e., the set of points such that a blowup is not a halfspace in $\R^N$, is no more than $N-3$ (see \cite[Theorem 2.19]{AntonelliPozzettaAlexandrov}). Differently from the smooth category, such estimate is sharp (cf. \cite[Remark 1.8]{MoS21}).\\
In the class of smooth Riemannian manifolds, higher regularity of isoperimetric sets is well-understood and boundaries of isoperimetric sets are smooth hypersurfaces out of a set of codimension $8$ \cite{morgan2003regularity}. A first example of a nonsmooth isoperimetric set has been recently given in \cite{NiuNonsmoothIsoperimetrico}.

\medskip

Taking into account \autoref{thm:Regularity}, from now on we will always assume that if $E$ is an isoperimetric set as in the assumptions of \autoref{thm:Regularity}, then $E=E^{(1)}$; in particular $E$ is open, bounded, $P(E,\cdot)=\haus^{N-1}\res \partial E$ and $E^{(0)}$ is the complement of $\overline{E}$.

\section{Properties of the isoperimetric profile on spaces with Ricci lower bounds}\label{sec:Main}

\subsection{Sharp differential inequalities of the isoperimetric profile}\label{sec:SharpDiffIneq}

Let $I\subset \R$ be an open interval and let $f:I\to \R$, $g:{\rm Im}\,(f)\to\R$ be continuous functions. We denote
\begin{equation}\label{eq:DefD2}
\begin{split}
\overline{D}^2f(x) & \eqdef \limsup_{h\to 0^+} \frac{f(x+h)+f(x-h)-2f(x)}{h^2}.
\end{split}
\end{equation}
Moreover we say that
\begin{itemize}
    \item $f''\le g(f)$ in the viscosity sense on $I$ if for any $x \in I$ and any smooth function $\varphi$ defined in a neighborhood of $x$ such that $\varphi - f$ has a local maximum at $x$, there holds $\varphi''(x)\le g(f(x))$;
    
    \item $f''\le g(f)$ in the sense of distributions on $I$ if
		\[
		\int f\varphi'' \de x \le \int g(f) \, \varphi \de x,
		\]
	for every $\varphi \in C^\infty_c(I)$ with $\varphi\ge 0$.
\end{itemize}
We recall that in the definition of viscosity solution it is equivalent to consider for any $x \in I$ smooth functions $\varphi$ defined in a neighborhood of $x$ such that $\varphi \le f$ and $\varphi(x)=f(x)$, see, e.g., \cite[Remark 5.6]{AmbrosioCarlottoMassaccesi}.

\medskip

The next result states the sharp differential inequalities satisfied by the isoperimetric profile.

\begin{theorem}[{Sharp differential inequalities of the isoperimetric profile \cite[Theorem 1.1]{APPSa}}]\label{thm:DifferentialInequalitiesProfile}
Let $N \in \N$ with $N\ge 2$, and let $K\in\R$. Let $(X,\dist,\haus^N)$ be an $\RCD(K,N)$ space. Assume that there exists $v_0>0$ such that $\haus^N(B_1(x)) \ge v_0$ for any $x \in X$.
Let $\psi\eqdef I_X^{\frac{N}{N-1}}$. Then $\psi$ solves
\begin{equation}\label{eq:DifferentialInequalitiesProfile}
    \psi'' \le - \frac{K\, N}{N-1} \psi^{\frac{2-N}{N}},
\end{equation}
in the sense of distributions on $(0,\haus^N(X))$, equivalently in the viscosity sense on $(0,\haus^N(X))$, equivalently $\overline{D}^2 \psi(V)\le - \tfrac{K\, N}{N-1} \psi^{\frac{2-N}{N}}(V)$ for any $V \in (0,\haus^N(X))$.
\end{theorem}

The previous theorem is sharp in the sense that \eqref{eq:DifferentialInequalitiesProfile} is an equality on simply connected manifolds with constant sectional curvature: spheres, Euclidean and hyperbolic spaces (see \eqref{eq:DefModels} below for a definition).

\medskip

As anticipated in the introduction, the coupling between a Ricci lower bound and an upper bound on the second derivative of the isoperimetric profile is a classical result, and it eventually relies on the second variation formula for the perimeter on Riemannian manifolds. In the smooth context, if $E\subset M$ is a smooth isoperimetric set on a Riemannian manifold $M$, the formula reads
\[
\frac{\d^2}{\d t^2} P(E_t)\bigg|_0 = \int_{\partial E} H^2 -\|{\rm II}\|^2 - \Ric(\nu_E,\nu_E),
\]
where $E_t$ denotes the $t$-tubular neighborhood of $E$\footnote{$E_t\eqdef \{x\in M\st \dist(x,E)<t\}$ for $t>0$, while $E_t\eqdef \{x \in E\st \dist(x,M\setminus E)>-t\}$ for $t\le0$.}, $H$ and ${\rm II}$ are the mean curvature and the second fundamental form of $E$, respectively, and $\Ric(\nu_E,\nu_E)$ denotes Ricci curvature of $M$ applied to the inner unit normal $\nu_E$ of $E$. A lower bound $\Ric\ge K$ on $M$ hence implies the upper bound
\[
\frac{\d^2}{\d t^2} P(E_t)\bigg|_0  \le \left(\frac{N-2}{N-1}H^2 - K\right)P(E),
\]
which, differentiating the composition $I_X\circ \haus^N(E_t)$, readily implies upper bounds on the second derivative of $I_X$. It is crucial to observe that the previous sketch relies on \emph{existence} of isoperimetric sets, as well as on the regularity of the isoperimetric profile and of the boundary of isoperimetric sets.

\medskip

The previous argument has its roots in \cite[Corollaire 6.6]{Gallotast} and \cite[Sect. 7]{BavardPansu86}. The observation of differentiating the $\tfrac{N}{N-1}$-power of the profile comes from \cite{KuwertIsop}. Differential inequalities for the profile have been proved in \cite{SternbergZumbrun} for the relative isoperimetric problem in bounded convex bodies, later generalized in \cite{BayleRosales} to manifolds with boundary and Ricci bounded below; in \cite{Bayle03, Bayle04}, after \cite[Sect. 2.1, Proposition 3.3]{MorganJohnson00}, \eqref{eq:DifferentialInequalitiesProfile} is proved on compact Riemannian manifolds with Ricci bounded below. In \cite{MorganMnfDensity, Bayle03, Milmanconvexity} analogous inequalities hold in the case of compact weighted manifolds. The derivation of \eqref{eq:DifferentialInequalitiesProfile} in the viscosity sense on compact manifolds also appears in \cite{NiWangiso}. To the author's knowledge, the first instances of differential inequalities for the profile of noncompact manifolds without assuming existence of isoperimetric sets are contained in \cite[Theorem 3.3]{MondinoNardulli16}, which however asks strong additional asymptotic conditions on the ambient, in \cite{LeonardiRitore} in the setting of Euclidean convex bodies, and in \cite[Theorem 1.4]{AntBruFogPoz}, which holds for manifolds satisfying just the assumptions of \autoref{thm:DifferentialInequalitiesProfile} but does not recover the sharp inequality \eqref{eq:DifferentialInequalitiesProfile}.

\medskip

In the next sections, we outline the proof of \autoref{thm:DifferentialInequalitiesProfile}. Even in case $X$ in \autoref{thm:DifferentialInequalitiesProfile} is a smooth noncompact manifold, the proof \emph{necessarily} exploits an analysis carried out over isoperimetric sets in $\RCD$ spaces possibly different from $X$, as existence of minimizers on $X$ is not guaranteed. In fact, the proof follows the following steps.
\begin{enumerate}
	\item The asymptotic mass decomposition result of \autoref{thm:AsymptoticMassDecomposition} identifies isoperimetric sets in limits at infinity along $X$, whose perimeter and measure are still related to the isoperimetric problem on the orginal space $X$.
	
	\item Recalling the regularity result in \autoref{thm:Regularity}, it is possible to codify a notion of ``constant mean curvature'' for the boundary of an isoperimetric set through Laplacian bounds satisfied by the distance function from the boundary of such set. This is the content of \autoref{thm:MeanCurvatureBarriers}.
	
	\item The previous Laplacian bounds imply Heintze--Karcher type estimates on volume and perimeter of tubular neighborhoods of the isoperimetric sets obtained in (1). These bounds allow to sharply estimate the upper second derivative of the composition of $I_X$ with the volume of tubular neighborhoods of the isoperimetric sets from (1), allowing to deduce \eqref{eq:DifferentialInequalitiesProfile} in the viscosity sense.
\end{enumerate}

\medskip

We conclude by mentioning that the validity of \autoref{thm:DifferentialInequalitiesProfile} is open in the collapsed case.

\begin{question}
    Let $N \in \N$ with $N\ge 2$, and let $K\in\R$. Let $(X,\dist,\meas)$ be a collapsed $\RCD(K,N)$ space. Under which hypotheses does \eqref{eq:DifferentialInequalitiesProfile} hold?
\end{question}

\subsection{Asymptotic mass decomposition}

The next result describes the general behavior of a minimizing sequence for the isoperimetric problem on a noncompact $\RCD(K,N)$ space $(X,\dist,\mathcal{H}^N)$. The mass of the sequence splits in finitely many pieces, each of them converging to isoperimetric sets in limits at infinity of $X$, except for at most one piece which converge to an isoperimetric set on $X$. The last set may have measure strictly lower than the one of the starting minimizing sequence, and possibly equal to zero.

\begin{theorem}[{Asymptotic mass decomposition, \cite[Theorem 1.1]{AntonelliNardulliPozzetta}}]\label{thm:AsymptoticMassDecomposition}
Let $N \in \N$ with $N\ge 2$, and let $K\in\R$. Let $(X,\dist,\mathcal{H}^N)$ be a noncompact $\RCD(K,N)$ space. Assume there exists $v_0>0$ such that $\mathcal{H}^N(B_1(x))\geq v_0$ for every $x\in X$. Let $V>0$. For every minimizing $($for the perimeter$)$ sequence of bounded sets $\Omega_i\subset X$ of volume $V$, up to passing to a subsequence, there exist a nondecreasing bounded sequence $\{N_i\}_{i\in\mathbb N}\subseteq \mathbb N$, disjoint sets of finite perimeter $\Omega_i^c, \Omega_{i,j}^d \subset \Omega_i$, and points $p_{i,j}$, with $1\leq j\leq N_i$ for any $i$, such that the following claims hold
\begin{itemize}
    \item $\lim_{i} \dist(p_{i,j},p_{i,\ell}) = \lim_{i} \dist(p_{i,j},o)=+\infty$, for any $j\neq \ell \le\overline N$ and any $o\in X$, where $\overline N:=\lim_i N_i <+\infty$;
    
    \item $\Omega_i^c$ converges to $\Omega\subset X$ in $L^1(X)$ and $ P( \Omega_i^c) \to_i P(\Omega)$. Moreover $\Omega$ is an isoperimetric region in $X$;
    
    \item for every $0<j\le\overline N$, $(X,\dist,\mathcal{H}^N,p_{i,j})$ converges in the pmGH sense  to an $\RCD(K,N)$ space $(X_j,\dist_j,\mathcal{H}^N,p_j)$. Moreover there exist isoperimetric regions $Z_j \subset X_j$ such that $\Omega^d_{i,j}\to_i Z_j$ in $L^1$-strong and $P(\Omega^d_{i,j}) \to_i P(Z_j)$;
    
    \item it holds that
    \begin{equation}\label{eq:UguaglianzeIntro}
    I_X(V) = P(\Omega) + \sum_{j=1}^{\overline{N}} P (Z_j),
    \qquad\qquad
    V=\mathcal{H}^N(\Omega) +  \sum_{j=1}^{\overline{N}} \mathcal{H}^N(Z_j).
    \end{equation}
\end{itemize}
\end{theorem}

By a standard truncation argument \cite[Lemma 2.17]{AFP21}, in the setting of \autoref{thm:AsymptoticMassDecomposition} it is always possible to choose a minimizing sequence for the isoperimetric problem made of bounded sets, thus such assumption on the sequence $\Omega_i$ is not restrictive.\\
The proof of \autoref{thm:AsymptoticMassDecomposition} follows by combining a concentration-compactness argument, see \cite[Lemma I.1]{Lions84I}, with the natural precompactness of $\RCD$ spaces, see \autoref{thm:Precompactness}, and of sequences of sets of finite perimeter, see \autoref{rem:SemicontPerimeterConverging}.
\autoref{thm:AsymptoticMassDecomposition} is inspired by the theory developed on Riemannian manifolds in \cite{RitRosales04, Nar14, MondinoNardulli16, FloresNardulli20, FloresNardulliCompactness}, where additional strong assumptions on the geometry at infinity of the space are assumed. Such assumptions where removed in \cite{AFP21}. Finally \cite{AntonelliNardulliPozzetta} contains the generalization to the $\RCD$ setting and the proof of the fact that $\overline{N}<+\infty$ in \autoref{thm:AsymptoticMassDecomposition}, which relies on the regularity \autoref{thm:Regularity}. Other recent analogous applications of the method to isoperimetric clusters are \cite{Reinaldo2020, NovagaClusters}. The coupling between the concentration-compactness principle and the precompactness of $\RCD$ spaces has been recently exploited also in \cite{NobiliVioloStability} to prove stability results for Sobolev inequalities.

\begin{remark}[Nontriviality of the problem]\label{rem:InfPositivo}
	Let $N \in \N$ with $N\ge 2$, let $K\in\R$, and $v_0,\overline{V}>0$. Then there exists $\mathscr{I}>0$ such that if $(X,\dist,\haus^N)$ is a noncompact $\RCD(K,N)$ space with $\inf_{x \in X}\haus^N(B_1(x))\ge v_0$, then $I_X(\overline{V}) \ge \mathscr{I}$.\\
	The claim easily follows by a contradiction argument applying \autoref{thm:AsymptoticMassDecomposition} together with the generalized compactness result from \cite[Theorem 1.2]{AntonelliNardulliPozzetta}. Alternatively, the observation follows adapting arguments from \cite[Theorem V.2.6]{ChavelIsoperimetricBook01}.\\
    We also remark that both a lower bound on the Ricci curvature and the noncollapsing condition given by $\inf_{x \in X} \haus^N(B_1(x))>0$ are necessary for the isoperimetric problem to make sense, i.e., for the isoperimetric profile to be strictly positive for positive volumes, see \cite[Sect. 4.3]{AFP21} and \cite[Proposition 2.18]{AntBruFogPoz} and the related \cite[Proposition 3.14]{LeonardiRitore} in the setting of convex bodies.
\end{remark}

\subsection{Mean curvature barriers}

We need to introduce some notation.

\begin{definition}\label{def:SignedDistance}
For any open set $E$ in a metric space $(X,\dist)$, the \emph{signed distance from $E$} is defined by
\[
\dist^s_E(p) \eqdef
\begin{cases}
\dist_E(p) & \text{ if } p \in X \setminus E,\\
-\dist_{X\setminus E}(p) & \text{ if } p \in E.
\end{cases}
\]
\end{definition}

Let $k,\lambda \in\R$. We define the function
\begin{equation*}
s_{k,\lambda}(r) \eqdef \cos_k(r) - \lambda \sin_k(r),
\end{equation*}
where $\cos_k$ and $\sin_k$ are the solution to the problems
\begin{equation}\label{eq:CosSin}
    \begin{cases}
    \cos_k''(r) + k \cos_k(r)=0,\\
    \cos_k(0)=1,\\
    \cos_k'(0)=0,
    \end{cases}
    \qquad\qquad
    \begin{cases}
    \sin_k''(r) + k \sin_k(r)=0,\\
    \sin_k(0)=0,\\
    \sin_k'(0)=1.
    \end{cases}
\end{equation}
For given $N\in \N$ with $N\ge 2$, the Riemannian manifold
\begin{equation}\label{eq:DefModels}
\mathbb{M}_k^N\eqdef \begin{cases}
([0,+\infty)\times \S^{N-1}, \d r^2 + \sin_k^2(r) g_{\S^{N-1}}) & \text{ if } k \le 0,\\
([0,\pi/\sqrt{k}]\times \S^{N-1}, \d r^2 + \sin_k^2(r) g_{\S^{N-1}}) & \text{ if } k>0,
\end{cases}
\end{equation}
is the (unique up to isometry) $N$-dimensional simply connected Riemannian manifold of constant sectional curvature $k$ (see \cite{Petersen2016}), where $g_{\S^{N-1}}$ is the standard metric on $\S^{N-1}$.
It is an exercise to check that the map
\begin{equation}\label{eq:zz}
r \mapsto (N-1) \frac{s'_{k,-\lambda}(r)}{s_{k,-\lambda}(r)}
\end{equation}
yields the mean curvature (with respect to inner normal) of the sphere defined by the points having signed distance $r$ from a ball whose boundary has mean curvature equal to $(N-1)\lambda$ in $\mathbb{M}^N_k$, for any $k,\lambda,r \in \R$ for which \eqref{eq:zz} makes sense.

\medskip

Before stating the next main result, we need to introduce the following generalized notion of Laplacian, that we state specialized to the noncollapsed $\RCD$ setting.

\begin{definition}[Measure Laplacian]
    Let $N \in \N$ with $N\ge 2$, and let $K\in\R$. Let $(X,\dist,\mathcal{H}^N)$ be an $\RCD(K,N)$ space. Let $\Omega\subset X$ be open. Let $F:X\to \R$ be a locally Lipschitz function. We say that $F$ has \emph{measure Laplacian on $\Omega$} if there exists a Radon measure $\mu$ on $\Omega$ such that for any Lipschitz function $f$ with compact support on $\Omega$, there holds
    \[
    \int \nabla F \cdot \nabla f \de \haus^N = - \int f \de \mu.
    \]
    In such a case we write $\boldsymbol{\Delta}F = \mu$. We say that $\boldsymbol{\Delta}F \ge G$ on $\Omega$, for some $G \in L^1_{\rm loc}(\Omega)$, if $F$ has measure Laplacian on $\Omega$ and
    \[
    -\int \nabla F \cdot \nabla f \de \haus^N \ge \int f \,G\de \haus^N,
    \]
    for any nonnegative Lipschitz function $f$ with compact support on $\Omega$.
\end{definition}

\begin{theorem}[{Existence of mean curvature barriers \cite[Theorem 1.3]{APPSa}}] \label{thm:MeanCurvatureBarriers}
Let $N \in \N$ with $N\ge 2$, and let $K\in\R$. Let $(X,\dist,\mathcal{H}^N)$ be an $\RCD(K,N)$ space. Let $E\subset X$ be an isoperimetric region.
Then, denoting by $f$ the signed distance function from $E$, there exists $c\in\setR$ such that
\begin{equation}\label{eq:LaplacianBounds}
\boldsymbol{\Delta} f\ge  -(N-1)\frac{s'_{\frac{K}{N-1},\frac{c}{N-1}}\circ \left(-f\right)}{s_{\frac{K}{N-1},\frac{c}{N-1}}\circ \left(-f\right)}      \quad\text{on $E$, }\qquad
\boldsymbol{\Delta} f\le (N-1)\frac{s'_{\frac{K}{N-1},-\frac{c}{N-1}}\circ f}{s_{\frac{K}{N-1},-\frac{c}{N-1}}\circ f}    \quad\text{on $X\setminus \overline{E}$}\, .
\end{equation}
If $K=0$, then $c\ge0$ and \eqref{eq:LaplacianBounds} reads
\begin{equation}\label{eq:LaplacianBounds0}
 \boldsymbol{\Delta} f\ge  \frac{c}{1+\frac{c}{N-1}f}  \quad\text{on $E$, }
 \qquad
 \boldsymbol{\Delta} f\le \frac{c}{1+\frac{c}{N-1}f} \quad\text{on $X\setminus \overline{E}$}.
\end{equation}
\end{theorem}

In the setting of \autoref{thm:MeanCurvatureBarriers}, any constant $c$ satisfying \eqref{eq:LaplacianBounds} is said to be a \emph{mean curvature barrier} for $E$. In case $X$ in \autoref{thm:MeanCurvatureBarriers} is a Riemannian manifold, then $c$ is the constant mean curvature of the (regular part of the) boundary of $E$; moreover \eqref{eq:LaplacianBounds} is a well-known consequence of the evolution of the Laplacian of the distance function along geodesics from the boundary, together with the classical regularity theory for boundaries of isoperimetric sets on smooth manifolds. Such classical argument seems out of reach in the nonsmooth setting. Instead, \autoref{thm:MeanCurvatureBarriers} follows exploiting the minimality of the set $E$ in \autoref{thm:MeanCurvatureBarriers} and the equivalence between distributional and viscosity bounds on
the Laplacian from \cite{MoS21}, recently generalized to the collapsed case in \cite{GigliMondinoSemolaLaplacianBounds}. In \cite{APPSa}, \eqref{eq:LaplacianBounds} is proved by contradiction, showing that if the bounds fail there exists a volume fixing perturbation of the set with strictly smaller perimeter, a contradiction with the isoperimetric
condition. The perturbations are built by sliding simultaneously level sets of distance-like functions with well controlled Laplacian, obtained by Hopf--Lax duality on test functions given by the absurd assumption that \eqref{eq:LaplacianBounds} does not hold in the viscosity sense. The method was firstly employed in \cite{MoS21} to show Laplacian bounds for the signed distance function from perimeter minimizers, and it was inspired by \cite{CaffarelliCordoba93, Petruninharmonic}.

\medskip

We mention that an analogous notion of constant mean curvature has been independently considered in \cite{Ketterer21}, while a different notion of mean curvature is introduced in \cite{Ketterer20}. See also \cite{LahtiShanmSpeight} for a further notion of domain with boundary of positive mean curvature in metric measure spaces.
We remark that \autoref{thm:MeanCurvatureBarriers} only shows existence of some mean curvature barrier, leaving the following question open.

%Commenti non unicità, $c\ge0$ su $K=0$ non compatto, $c=0$ splitti \cite{Kasue83, Ketterer21}.

\begin{question}\label{question:UnicitaBarriera}
In the assumptions of \autoref{thm:MeanCurvatureBarriers}, if $E\subset X$ is isoperimetric, does there exist a unique number $c\in\R$ such that \eqref{eq:LaplacianBounds} holds?
\end{question}

A positive answer to \autoref{question:UnicitaBarriera} would allow to speak of constant mean curvature for isoperimetric sets. The answer to \autoref{question:UnicitaBarriera} is negative if $E$ is not an isoperimetric set, see \cite[Remark 3.9]{APPSa}. Clearly, the latter nonuniqueness phenomenon is a feature of the nonsmooth framework.

\medskip

Integrating \eqref{eq:LaplacianBounds} over tubular neighborhoods of isoperimetric sets and exploiting the Gauss--Green formula \cite[Theorem 1.6]{BPSGaussGreen}, one readily finds Heintze--Karcher type inequalities on perimeter and volume of such tubular neighborhoods \cite{HeintzeKarcher}. We state such a consequence in the case $K=0$ only.

\begin{corollary}\label{cor:HK}
	Let $N \in \N$ with $N\ge 2$. Let $(X,\dist,\mathcal{H}^N)$ be an $\RCD(0,N)$ space. Let $E\subset X$ be an isoperimetric region, let $f$ be the signed distance function from $E$, and let $c\in[0,+\infty)$ be a mean curvature barrier for $E$. Denoting $E_t \eqdef \{ x \st f(x) <t \}$, for $t \in \R$, there holds
	\begin{equation*}
		P(E_t) \le P(E) \left(1 + \frac{c}{N-1}t \right)^{N-1},
		\qquad
		|\haus^N(E_t) - \haus^N(E)| \le  P(E) \int_0^{|t|} \left(1 + {\rm sgn}(t)\,\frac{c}{N-1}s \right)^{N-1}_+ \de s,
	\end{equation*}
    where $(\cdot)_+$ denotes positive part.
\end{corollary}

\autoref{cor:HK} immediately implies the next observation.

\begin{remark}\label{rem:AVRcPositiva}
    Let $N \in \N$ with $N\ge 2$. Let $(X,\dist,\mathcal{H}^N)$ be an $\RCD(0,N)$ space and fix $o \in X$.  Let $E\subset X$ be an isoperimetric region. If $\liminf_{r\to+\infty}\haus^N(B_r(o))/r =+\infty$, then any mean curvature barrier $c$ for $E$ is strictly positive.\\
    In fact, if an isoperimetric set has mean curvature barrier equal to zero, it is possible to prove that the complement of $E$ is made of cylindrical ends, see \cite[Theorem 4.11]{Ketterer21}, \cite{KettererKitabeppuLakzian}, and \cite[Theorem 2.18]{AntonelliPozzettaAlexandrov} for a precise statement.
\end{remark}

\subsection{Proof of \autoref{thm:DifferentialInequalitiesProfile} and some consequences}

\begin{proof}[Proof of \autoref{thm:DifferentialInequalitiesProfile}]
We give a proof in the case $K=0$ only and assuming that $X$ is noncompact. We shall prove that \eqref{eq:DifferentialInequalitiesProfile} holds in the viscosity sense, equivalence with other formulations follow from \autoref{prop:EquivalenceDifferentialInequalities}. Let $V \in (0,\haus^N(X))$ and let $\varphi$ be a smooth function defined in a neighborhood of $V$ such that $\varphi\le I_X$ and $\varphi(V)= I_X(V)$. Let us apply the asymptotic mass decomposition \autoref{thm:AsymptoticMassDecomposition} at volume $V$. Let us assume for simplicity that \autoref{thm:AsymptoticMassDecomposition} yields exactly one isoperimetric set $E\subset Y$ with $P(E)=I_X(\haus^N(E))$ of measure $V$ contained in an $\RCD(0,N)$ space $(Y,\dist_Y,\haus^N)$ which either coincides with $X$ or is a pmGH limit at infinity of $X$; the general case follows by minor adaptations. By \autoref{thm:MeanCurvatureBarriers}, $E$ has a mean curvature barrier $c$. For $E_t$ as in \autoref{cor:HK}, we estimate
\begin{equation}\label{eq:zwz}
    \varphi(\haus^N(E_t)) \le I_X(\haus^N(E_t)) \le I_Y(\haus^N(E_t)) \le P(E_t) \le  P(E) \left(1 + \frac{c}{N-1}t \right)^{N-1},
\end{equation}
for $t$ close to zero, where in the second inequality we used the fact that the profile of any limit at infinity of $X$ is no less than the isoperimetric profile on $X$, which follows by an immediate approximation argument (see \cite[Proposition 2.19]{AntonelliNardulliPozzetta}, \cite[Proposition 3.2]{AFP21}, and recall \autoref{rem:SemicontPerimeterConverging}). Since equality holds in \eqref{eq:zwz} for $t=0$, then there exists the derivative $P(E_t)'|_{t=0} = c\,P(E)$ and we get that $\varphi'(V) = c$. Hence there exists the second derivative $\haus^N(E_t)''|_{t=0}= c \, P(E) = c \, I_X(V)$, and thus
\begin{equation}\label{eq:zz1}
    (\varphi(\haus^N(E_t)))''|_{t=0} = \varphi''(V) I_X^2(V) + (\varphi'(V))^2 \, I_X(V).
\end{equation}
On the other hand by \autoref{cor:HK} we estimate
\begin{equation}\label{eq:zz2}
\begin{split}
    (\varphi(\haus^N(E_t)))''|_{t=0} &= \lim_{t\to0^+} \frac{\varphi(\haus^N(E_t)) + \varphi(\haus^N(E_{-t})) - 2 \varphi(V)}{t^2} \le \limsup_{t\to0^+} \frac{P(E_t) + P(E_{-t}) - 2 P(E)}{t^2}\\
    &\le P(E) \frac{\d^2}{\d^2t} \left(1 + \frac{c}{N-1}t \right)^{N-1} \bigg|_{t=0} = \frac{N-2}{N-1} (\varphi'(V))^2 \, I_X(V).
\end{split}
\end{equation}
Putting together \eqref{eq:zz1} and \eqref{eq:zz2} we find $\varphi''(V) I_X(V) \le - (\varphi'(V))^2 /(N-1)$. This shows that $I_X$ solves $I_X''\,I_X \le  - (I_X')^2 /(N-1)$ in the viscosity sense on $(0,\haus^N(X))$. Recalling \autoref{rem:InfPositivo}, this is readily checked to be equivalent to the desired \eqref{eq:DifferentialInequalitiesProfile}.
\end{proof}

We collect a couple of useful observations following from \autoref{thm:DifferentialInequalitiesProfile} that we shall use in the following.

\begin{remark}\label{rem:DerivataProfiloBarriera}
    Let $N \in \N$ with $N\ge 2$, and let $K\in\R$. Let $(X,\dist,\mathcal{H}^N)$ be a noncompact $\RCD(K,N)$ space. Assume there exists $v_0>0$ such that $\mathcal{H}^N(B_1(x))\geq v_0$ for every $x\in X$. Let $\overline{V}>0$ and let $\Omega, Z_1,\ldots, Z_{\overline{N}}$ be isoperimetric sets obtained by applying \autoref{thm:AsymptoticMassDecomposition} at the volume $\overline{V}$.
    If $I_X$ is differentiable at $\overline{V}$, then any of the sets $\Omega, Z_1,\ldots, Z_{\overline{N}}$ has a unique mean curvature barrier equal to $I_X'(\overline{V})$.\\
    The previous claim easily follows differentiating the perimeter of tubular neighborhoods as in the proof of \autoref{thm:DifferentialInequalitiesProfile}.
\end{remark}

\begin{remark}\label{rem:UnicoPezzo}
    Let $N \in \N$ with $N\ge 2$. Let $(X,\dist,\mathcal{H}^N)$ be a noncompact $\RCD(0,N)$ space. Assume there exists $v_0>0$ such that $\mathcal{H}^N(B_1(x))\geq v_0$ for every $x\in X$.
    Then applying \autoref{thm:AsymptoticMassDecomposition} to any perimeter minimizing sequence on $X$ gives either $\overline{N}=0$ or $\overline{N}=1$, i.e., either all the mass remains in $X$ or all the mass escapes to a unique limit at infinity.\\
    The previous claim readily follows since in the case of nonnegative curvature, \autoref{thm:DifferentialInequalitiesProfile} implies that the isoperimetric profile is strictly subadditive. 
\end{remark}

In the next remark, we point out further information which can be deduced out of \autoref{thm:DifferentialInequalitiesProfile} on regularity properties for the isoperimetric profile and isoperimetric sets.

\begin{remark}[{Improved regularity of isoperimetric profile and sets \cite{APPSa}}]\label{rem:RegularityNew}
Let $N \in \N$ with $N\ge 2$, and let $K\in\R$. Let $(X,\dist,\mathcal{H}^N)$ be a noncompact $\RCD(K,N)$ space. Assume there exists $v_0>0$ such that $\mathcal{H}^N(B_1(x))\geq v_0$ for every $x\in X$. Then the following holds.
\begin{itemize}
    \item For any $0<V_1<V_2$ there is $L=L(K,N,v_0,V_1,V_2)$ such that $I_X$ is $L$-Lipschitz on $[V_1,V_2]$.

    \item Let $E\subset X$ be an isoperimetric set. Then $E$ satisfies uniform volume and perimeter density estimates at boundary points with constant depending only on $K,N,v_0,\haus^N(E)$.
\end{itemize}
We stress that the previous regularity properties are independent of the specific ambient $X$. This allows to derive a strong stability result for sequences of isoperimetric sets in a sequence of $\RCD(K,N)$ spaces satisfying a uniform lower bound on the volume of unit balls, see \cite{APPSa}. For example, if $E_i$ is a sequence of isoperimetric sets of volume $1$ in a sequence of unbounded pointed $\RCD(K,N)$ spaces $(X_i,\dist_i, \haus^N,x_i)$, $\inf_{x \in X_i}\haus^N(B_1(x)) \ge v_0$ and $E_i\subset B_R(x_i)$, for some $v_0,R>0$, then the sequence is precompact with respect to $L^1$-strong and Hausdorff convergence (in a realization), any limit set is isoperimetric, and mean curvature barriers of any $E_i$ are uniformly bounded and converge to mean curvature barriers of limit sets. We mention that a similar stability result for mean curvature barriers was observed in \cite{Ketterer21}.
\end{remark}

We conclude with a final question related to \autoref{thm:DifferentialInequalitiesProfile}. We consider the case $K=0$ only.

\begin{question}\label{question:InverseInequalities}
    Let $N \in \N$ with $N\ge 2$. Let $\mathscr{C}$ be the class of concave functions $\psi:[0,+\infty)\to [0,+\infty)$ such that $\psi(0)=0$ and $\psi(V)>0$ for $V>0$. Let $\mathscr{R}$ be the class of noncompact $\RCD(0,N)$ spaces $(X,\dist,\haus^N)$ such that $\inf_{x \in X} \haus^N(B_1(x))>0$. Let $\Psi: \mathscr{R} \to \mathscr{C}$ be the map $\Psi(X) \eqdef I_X^{\frac{N}{N-1}}$. Is $\Psi$ injective? What is the image of $\Psi$? What is the image of $\Psi$ restricted to smooth manifolds?
\end{question}

The previous question is also related to the investigation of further regularity properties of isoperimetric profiles, for instance a better understanding of their differentiability.

Obviously, \autoref{question:InverseInequalities} can be suitably stated for arbitrary Ricci lower bounds and in the setting of compact spaces.

\section{Isoperimetric inequalities and existence results on nonnegatively curved spaces}\label{sec:Applications}

This section is devoted to the proof of old and new results, making a fundamental use of the differential inequalities of the profile from \autoref{thm:DifferentialInequalitiesProfile}.

\subsection{Geometric analysis on spaces with Ricci lower bounds}\label{sec:GeomAnal}

We recall some fundamental results in the theory of the Riemannian Geometry of manifolds with Ricci bounded below which naturally extend to the context of $\RCD(K,N)$ spaces.

\medskip

Let $K\in  \R$ and $N \in \N$ with $N\ge 2$. Recalling \eqref{eq:CosSin}, for $r\ge0$ we denote by
\[
\sigma_{K,N}(r) \eqdef N\omega_N \sin_{\frac{K}{N-1}}^{N-1} (r) ,
\qquad\qquad 
v_{K,N}(r) \eqdef N \omega_N \int_0^r \sin_{\frac{K}{N-1}}^{N-1}(t) \de t,
\]
the perimeter and the volume volume of a ball of radius $r$ in the simply connected $N$-dimensional Riemannian manifold of constant sectional curvature $K/(N-1)$ (here $r \le \pi\sqrt{N-1}/\sqrt{K}$ if $K>0$), where $\omega_N$ is the volume of the Euclidean unit ball of $\R^N$ (see \cite{Petersen2016} and \eqref{eq:DefModels}).\\
We can state the $\RCD$ version of the Bishop--Gromov monotonicity of volume ratios with rigidity.

\begin{theorem}[{Generalized Bishop--Gromov monotonicity and rigidity, \cite[Theorem 5.1]{OhtaMCP} \& \cite[Theorem 2.3]{Sturm2} \& \cite[Theorem 1.1, Theorem 4.1]{DePhilippisGigli16}}]\label{thm:BishopGromov}
Let $N \in \N$ with $N\ge 2$, and $K \in \R$. Let $(X,\dist,\haus^N)$ be an $\RCD(K,N)$ space. Let $o \in X$. Then
\[
\begin{drcases}
(0,+\infty) & \text{ if } K\le0,\\
(0,\pi\sqrt{N-1}/\sqrt{K}] & \text{ if } K >0
\end{drcases} \ni r \mapsto \frac{\haus^N(B_r(o))}{v_{K,N}(r)},
\]
is nonincreasing, and
\[
\begin{drcases}
(0,+\infty) & \text{ if } K\le0,\\
(0,\pi\sqrt{N-1}/\sqrt{K}] & \text{ if } K >0
\end{drcases} \ni r \mapsto \frac{P(B_r(o))}{\sigma_{K,N}(r)},
\]
is essentially nonincreasing, i.e., $P(B_r(o))/\sigma_{K,N}(r) \ge P(B_R(o))/\sigma_{K,N}(R)$ for any $R>0$ such that the ratio is defined, and for a.e. $r\in(0,R]$.\\
Moreover
\begin{itemize}
    \item if $K=0$ and $\haus^N(B_R(o))/R^N=\haus^N(B_r(o))/r^N$ for some $0<r<R$, then $(B_{\frac{R}{2}}(o),\dist)$ is isometric to the ball of radius $R/2$ centered at a tip in the cone $C(L)$ over an $\RCD(N-2,N-1)$ space $(L,\dist_L,\haus^{N-1})$;
    
    \item if $K=N-1$ and $\haus^N(B_R(o))/v_{N-1,N}(R)=\haus^N(B_r(o))/v_{N-1,N}(r)$ for some $0<r<R$, then $(B_{\frac{R}{2}}(o),\dist)$ is isometric to the ball of radius $R/2$ centered at a pole in the spherical suspension $S(L)$ over an $\RCD(N-2,N-1)$ space $(L,\dist_L,\haus^{N-1})$.
\end{itemize}
\end{theorem}

We remark that the monotonicity part in \autoref{thm:BishopGromov} also holds in the more general frameworks of $\CD(K,N)$ spaces \cite[Theorem 2.3]{Sturm2} and of $\mathsf{MCP}(K,N)$ spaces \cite[Theorem 5.1]{OhtaMCP}.

\medskip

A version of the classical maximal diameter Bonnet--Myers Theorem \cite{Petersen2016} holds at the level of $\RCD$ spaces.

\begin{theorem}[{Generalized Bonnet--Myers Theorem, \cite[Theorem 4.3]{OhtaMCP} \& \cite[Theorem 1.4]{Ketterer15}}]\label{thm:BonnetMyers}
Let $N \in \N$ with $N\ge 2$. Let $(X,\dist,\haus^N)$ be an $\RCD(N-1,N)$ space. Then $\diam(X) \le \pi$. Moreover equality holds if and only if $X$ is isometric to the spherical suspension over an $\RCD(N-2,N-1)$ space. In particular, if $X$ is a Riemannian manifold, equality holds if and only if $X$ is isometric to $\S^N$.
\end{theorem}

We remark that the upper bound on the diameter in \autoref{thm:BonnetMyers} actually holds in the greater generality of $\mathsf{MCP}(K,N)$ spaces, see \cite[Theorem 4.3]{OhtaMCP}.

\medskip

We mention that two further classical results in Riemannian geometry holding at the level of $\RCD(K,N)$ spaces are given by the Laplacian comparison theorem, see \cite[Corollary 5.15]{GigliOnTheDiffStructure}, and the Cheeger--Gromoll splitting theorem \cite{CheegrGromollSplitting}, see \cite[Theorem 1.4]{Gigli13}.

% Finally we recall an $\RCD$ version of the celebrated Cheeger--Gromoll splitting theorem \cite{CheegrGromollSplitting}.

% \begin{theorem}[{Generalized splitting theorem, \cite[Theorem 1.4]{Gigli13}}]\label{thm:SplittingTheorem}
% Let $N \in \N$ with $N\ge 2$. Let $(X,\dist,\haus^N)$ be an $\RCD(0,N)$ space. Suppose that $X$ contains a line, i.e., there exists a geodesic $\gamma:\R\to X$. Then $X$ is isometric to a product $(\R\times Y, \dist_{\rm eu}\otimes \dist_Y , \mathcal{L}^1 \otimes \haus^{N-1})$, where $(Y,\dist_Y, \haus^{N-1})$ is $\RCD(0,N-1)$.
% \end{theorem}

% Once again, we mention that the splitting theorem holds among possibly collapsed $\RCD(0,N)$ space, see \cite{Gigli13}.

\subsection{Sharp and rigid isoperimetric inequality on spaces with nonnegative Ricci and Euclidean volume growth}

Exploiting the differential inequalities of the isoperimetric profile, we prove the sharp isoperimetric inequality on $\RCD(0,N)$ spaces $(X,\dist,\haus^N)$ with Euclidean volume growth, together with a proof of its rigidity in this class.

\begin{definition}
Let $N \in \N$ with $N\ge 2$. Let $(X,\dist,\haus^N)$ be an $\RCD(0,N)$ space. Let $o \in X$. We define the \emph{asymptotic volume ratio} by
\[
\AVR(X)\eqdef \lim_{r\to+\infty} \frac{\haus^N(B_r(o))}{\omega_Nr^N}.
\]
\end{definition}

The previous definition is clearly independent of the choice of $o \in X$ and it is well-posed by \autoref{thm:BishopGromov}. For an $\RCD(0,N)$ space $(X,\dist,\haus^N)$, \autoref{thm:BishopGromov} and the fact that the volume density $x \mapsto \lim_{r\searrow0} \haus^N(B_r(x))/(\omega_Nr^N)$ is lower semicontinuous - hence, $\le 1$ - implies that $\AVR(X) \in [0,1]$. Moreover $\AVR(X)=1$ if and only if $X$ is isometric to the Euclidean space $\R^N$ endowed with Lebesgue measure (see \cite[Theorem 1.6]{DePhilippisGigli18} after \cite{Colding97}). If $\AVR(X)>0$ we say that $X$ has \emph{Euclidean volume growth}.

We are ready for the main result of the section.

\begin{theorem}\label{thm:DisugisoperimetricaAVR}
Let $N \in \N$ with $N\ge 2$. Let $(X,\dist,\haus^N)$ be an $\RCD(0,N)$ space with $\AVR(X)>0$. Then
\begin{equation}\label{eq:DisugIsoperimetricaAVR}
I_X(V) \ge N(\AVR(X)\omega_N)^{\frac1N} \,V^{\frac{N-1}{N}},
\end{equation}
for any $V>0$. Moreover if there exists an isoperimetric set $E\subset X$ such that
\[
P(E) = N(\AVR(X)\omega_N)^{\frac1N} \,(\haus^N(E))^{\frac{N-1}{N}},
\]
then $X$ is isometric to a cone $C$ over an $\RCD(N-2,N-1)$ space and $E$ is isometric to a ball centered at a tip of $C$. In particular, if $X$ is a Riemannian manifold, such an isoperimetric set exists if and only if $X$ is isometric to $\R^N$ and $E$ is isometric to a Euclidean ball.
\end{theorem}

Inequality \eqref{eq:DisugIsoperimetricaAVR} has been proved in several recent works at different levels of generality. In \cite{AgostinianiFogagnoloMazzieri} it was proved in the class of $3$-dimensional Riemannian manifolds with nonnegative Ricci and $\AVR>0$, later extended up to dimension $7$ in \cite{FogagnoloMazzieri}, exploiting nonlinear potential theory and the relation between perimeter and capacities; in \cite{BrendleFigo} the inequality was proved on any Riemannian manifold with nonnegative Ricci and $\AVR>0$ by using an ABP-type argument; \cite{Johne} provides a proof analogous to the one in \cite{BrendleFigo} in the class of weighted manifolds with nonnegative generalized Ricci curvature and positive asymptotic volume ratio; in \cite{BaloghKristaly} (resp., \cite{CavallettiManini}) the inequality is proved in the class of $\CD(0,N)$ (resp., $\mathsf{MCP}(0,N)$) spaces with $\AVR>0$ exploiting the Brunn--Minkowski inequality (resp., $1$-dimensional localization technique); \cite{APPSb} contains a proof for noncollapsed $\RCD(0,N)$ spaces with $\AVR>0$ and exploits  \autoref{thm:AsymptoticMassDecomposition} and \autoref{thm:MeanCurvatureBarriers}.\\
Furthermore, rigidity for the inequality is proved in \cite{AgostinianiFogagnoloMazzieri, BrendleFigo, FogagnoloMazzieri} in the class of smooth sets, in \cite{CavallettiManiniRigidity} in $\CD(0,N)$ spaces in the class of bounded sets, in \cite{APPSb} in the class of noncollapsed $\RCD(0,N)$ spaces for any isoperimetric set. Finally, exploiting \cite{CavallettiManiniRigidity}, the regularity result in \cite{AntonelliPasqualettoPozzettaViolo} implies rigidity for the isoperimetric inequality in the whole class of possibly collapsed $\RCD(0,N)$ spaces without restrictions on the isoperimetric set achieving equality.\\
We mention also that isoperimetric inequalities in the same spirit of \autoref{thm:DisugisoperimetricaAVR} are proved in \cite{HanIsoperimetric0Infty} for $\RCD(0,\infty)$ spaces with finite volume entropy, and in \cite{ManiniFinsler} for Finsler manifolds. A pioneering study of isoperimetric inequalities in terms of volume growth was carried out in \cite{CoulhonSaloffCoste}.

\medskip

We remark that \eqref{eq:DisugIsoperimetricaAVR} is sharp on \emph{any} space $X$ as in the assumptions of \autoref{thm:DisugisoperimetricaAVR}; in fact, \eqref{eq:DisugIsoperimetricaAVR} is seen to be sharp for large volumes. Indeed the ratio $P(B_r(o))/\haus^N(B_r(o))^{\frac{N-1}{N}}$ tends to $N(\AVR(X)\omega_N)^{\frac1N}$ as $r\to+\infty$ by \autoref{thm:BishopGromov}.

\medskip

A direct corollary of \autoref{thm:DisugisoperimetricaAVR} is the solution to the isoperimetric problem on noncollapsed $\RCD(0,N)$ cones, recovering the result from \cite{LionsPacella} on Euclidean convex cones (recall \autoref{rem:RelativePerimeter}) and the one from \cite{MorganRitore02} for cones over compact manifolds.

\begin{corollary}\label{cor:IsoperimetricaSuConi}
    Let $N \in \N$ with $N\ge 2$. Let $(C,\dist,\haus^N)$ be an $\RCD(0,N)$ cone over an $\RCD(N-2,N-1)$ space. Then
    \[
    I_C(V) = N(\AVR(C)\omega_N)^{\frac1N} \,V^{\frac{N-1}{N}},
    \]
    for any $V>0$, and balls centered at tips of $C$ are the only isoperimetric sets.
\end{corollary}

We now provide a proof of \autoref{thm:DisugisoperimetricaAVR} which partly simplifies the one from \cite{APPSb}, exploiting \autoref{thm:DifferentialInequalitiesProfile} more directly, as we shall do for the proof of the L\'{e}vy--Gromov inequality below.

\begin{proof}[Proof of \autoref{thm:DisugisoperimetricaAVR}]
Denote $\theta\eqdef \AVR(X)$ for ease of notation. Let $f_1(V)\eqdef N^{\frac{N}{N-1}}(\theta\omega_N)^{\frac{1}{N-1}} \,V$ and $f_2\eqdef I_X^{\frac{N}{N-1}}$. Suppose by contradiction that there is $V_0>0$ such that $f_1(V_0)>f_2(V_0)$. Since $f_2$ is concave by \autoref{thm:DifferentialInequalitiesProfile} and \autoref{lem:Concavity}, then there exists $v_1 \in (0,V_0)$ such that $I_X$ is differentiable at $v_1$ and $f_2'(v_1)<f_1'\equiv N^{\frac{N}{N-1}}(\theta\omega_N)^{\frac{1}{N-1}}$. By \autoref{rem:InfPositivo} and concavity, there holds $I'_X(v_1)\ge 0$. Let us apply the mass decomposition \autoref{thm:AsymptoticMassDecomposition} at volume $v_1$. By \autoref{rem:UnicoPezzo}, we get the existence of an isoperimetric set $E \subset Y$ with $P(E)=I_X(v_1)$ and measure equal to $v_1$, where $(Y,\dist_Y,\haus^N)$ either coincides with $X$ or is a limit at infinity of $X$. Observe that $\AVR(Y)\ge \theta$ by \autoref{thm:BishopGromov} and \autoref{thm:VolumeConvergence}. Hence by \autoref{rem:DerivataProfiloBarriera} and \autoref{rem:AVRcPositiva}, $E$ has mean curvature barrier $c=I_X'(v_1)  >0$. Thus \autoref{cor:HK} implies
\begin{equation}\label{eq:zzz}
\haus^N(E_t) \le \haus^N(E) + P(E) \frac{N-1}{c\,N} \left[ \left( 1+ \frac{c}{N-1}t\right)^N -1 \right]
\end{equation}
where $E_t\eqdef \{ y \in Y \st \dist_Y(y,E)< t\}$ for $t>0$. Condition $f_2'(v_1)<f_1'$ is rewritten as
\begin{equation}\label{eq:zzz2}
    \left( \frac{N}{N-1} \right)^{N-1} P(E) c^{N-1} < N^N\,\theta \omega_N.
\end{equation}
Dividing \eqref{eq:zzz} by $t^N$, exploiting \eqref{eq:zzz2}, and recalling $\AVR(Y)\ge \theta$, letting $t\to+\infty$ we get the contradiction
\[
\theta \omega_N \le \limsup_{t\to+\infty} \frac{\haus^N(E_t) }{t^N} < \theta \omega_N.
\]
Suppose now that there exists an isoperimetric set $E\subset X$ such that $P(E) = N(\theta\omega_N)^{\frac1N} \,(\haus^N(E))^{\frac{N-1}{N}}$. Denote $\overline{V}\eqdef \haus^N(E)$, hence $f_1(\overline{V})= f_2(\overline{V})$. Since $f_2\ge f_1$ by \eqref{eq:DisugIsoperimetricaAVR}, $f_2$ is concave, and $f_1(0)=f_2(0)=0$, then $f_1\equiv f_2$. Thus $I_X$ is differentiable and $E$ has mean curvature barrier $c = I_X'(\haus^N(E)) = (N-1)(\theta\omega_N)^{\frac{1}{N}} \haus^N(E)^{-\frac1N} = \frac{N-1}{N} \frac{P(E)}{\haus^N(E)}$. The mean curvature barrier's equation \eqref{eq:LaplacianBounds0} implies that $\sup_{x \in E} \dist(x,X\setminus E) \le (N-1)/c$, otherwise the inequality degenerates. Hence we use \autoref{cor:HK} again to get
\[
\haus^N(E) \le  P(E) \int_0^{\frac{N-1}{c}} \left( 1- \frac{c}{N-1}s\right)^{N-1}\de s = P(E) \frac{N-1}{c \, N} = \haus^N(E).
\]
Hence some ball $B_{\frac{N-1}{c}}(x_0)$ is contained in $E$. Since $\haus^N(B_{\frac{N-1}{c}}(x_0))\le \haus^N(E) = \theta \omega_N (\frac{N-1}{c})^N$, Bishop--Gromov \autoref{thm:BishopGromov} implies that $E=B_{\frac{N-1}{c}}(x_0)$, $\haus^N(B_R(x_0)) =\theta\omega_N R^N$ for any $R\ge (N-1)/c$ and rigidity follows.
\end{proof}

Carefully reading the previous proof of \autoref{thm:DisugisoperimetricaAVR}, we notice that, when $X$ is as in the assumptions, if $I_X(\overline{V})=  N(\AVR(X)\omega_N)^{\frac1N} \,(\overline{V})^{\frac{N-1}{N}}$ for some $\overline{V}>0$, then $I_X(V) = N(\AVR(X)\omega_N)^{\frac1N} \,(V)^{\frac{N-1}{N}} $ for any $V$ just by concavity, without needing the existence of an isoperimetric set. By \autoref{thm:DisugisoperimetricaAVR}, on such a space there exists an isoperimetric set if and only if it is a cone.

\begin{question}\label{question:UguaglianzaProfilo}
Does there exist an $\RCD(0,N)$ space $(X,\dist,\haus^N)$ with $\AVR(X)>0$ such that $I_X(\overline{V})=  N(\AVR(X)\omega_N)^{\frac1N} \,(\overline{V})^{\frac{N-1}{N}}$ for some $\overline{V}>0$ (hence, for any $V$) and $X$ is not a metric cone? Does there exist such a space $X$ in the class of smooth manifolds?
\end{question}

By \autoref{thm:AsymptoticMassDecomposition}, a space $X$ as in \autoref{question:UguaglianzaProfilo} has a pmGH limit at infinity isometric to a cone $C$ with $\AVR(C)=\AVR(X)$. Also, by \autoref{thm:DisugisoperimetricaAVR}, in a space $X$ as in \autoref{question:UguaglianzaProfilo} there do not exist isoperimetric sets.\\
Notice that \autoref{question:UguaglianzaProfilo} has a negative answer in the class of manifolds with nonnegative sectional curvature (or, more generally, Alexandrov spaces with nonnegative curvature, or $\RCD(0,N)$ spaces as in \autoref{thm:ABFPexistence} below) by \cite[Theorem 1.2]{APPSb}.

\subsection{Existence for large volumes on spaces with nonnegative Ricci and Euclidean volume growth}

In this section we provide some general existence results for large volumes of isoperimetric sets on nonnegatively curved spaces with Euclidean volume growth.

Here is the key observation. \autoref{thm:BishopGromov} and \autoref{thm:VolumeConvergence} readily imply that $\AVR$ is upper semicontinuous with respect to pmGH convergence of $\RCD(0,N)$ spaces. If we assume that pmGH limits at infinity of an $\RCD(0,N)$ space with $\AVR(X)>0$ have strictly larger $\AVR$ by a fixed gap $\eps>0$, it is possible to exploit the sharpness of the isoperimetric inequality \eqref{eq:DisugIsoperimetricaAVR} for large volumes to prove existence of isoperimetric sets for large volumes. Indeed, applying \autoref{thm:AsymptoticMassDecomposition} on a minimizing sequence of sufficiently large volume, this assumption implies that it is not isoperimetrically convenient to lose mass at infinity. This is the content of the next result.

\begin{theorem}[{\cite[Theorem 1.1, Theorem 1.2]{AntBruFogPoz} \& 
\cite[Theorem 1.2]{APPSb}}]\label{thm:ABFPexistence}
Let $N \in \N$ with $N\ge 2$. Let $(X,\dist,\haus^N)$ be an $\RCD(0,N)$ space with $\AVR(X)>0$. Write $X=\R^k\times X'$ for some $k\in\{0,\ldots,N-1\}$, where $X'$ does not contain lines.\\
Assume that there exists $\eps>0$ such that the following holds. For any sequence $(t_i,x'_i)\in X$ with $x'_i$ diverging along $X'$ such that $(X,\dist, \haus^N,(t_i,x'_i))$ pmGH converges, the pmGH limit $(Y,\dist_Y,\haus^N,y)$ satisfies
\[
\AVR(Y) \ge \AVR(X) + \eps.
\]
Then there exists $V_0$ such that for any $V\ge V_0$ there exists an isoperimetric region of volume $V$ on $X$.
\end{theorem}

The main assumption in the previous \autoref{thm:ABFPexistence} may appear artificial, though spontaneous, but it is actually strongly related to the geometry of the asymptotic cones of the space and to the stability of isoperimetric sets in such cones.

\begin{definition}\label{def:AsymptoticCone}
Let $N \in \N$ with $N\ge 2$. Let $(X,\dist,\haus^N)$ be an $\RCD(0,N)$ space with $\AVR(X)>0$. Fix $o \in X$ and let $\lambda_i\to+\infty$. Then any pmGH limit, up to subsequence, of $(X,\dist/\lambda_i, \haus^N , o)$ is an \emph{asymptotic cone} (or, \emph{tangent cone at infinity}) of $X$.
\end{definition}

\autoref{thm:Precompactness} and \autoref{thm:BishopGromov} ensure that the previous definition is well-posed. Moreover, an asymptotic cone is, indeed, a metric cone.
Observe that the sharp isoperimetric inequality in \autoref{thm:DisugisoperimetricaAVR} states that the profile of an $\RCD(0,N)$ space $X$ with Euclidean volume growth is bounded below by the one of any asymptotic cone to $X$.

\medskip

In general, asymptotic cones are not unique and they depend on the choice of the sequence $\lambda_i$ in \autoref{def:AsymptoticCone}, see \cite{PerelmanExampleCones, ColdingNaber, ColdingMinicozziConi}. Moreover, asymptotic cones to Riemannian manifolds can clearly be nonsmooth $\RCD(0,N)$ spaces.\\
More dramatically, asymptotic cones to a space $X$ may contain lines even $X$ does not, see \cite[pp. 913-914]{KasueWashio} and \cite[Theorem 1.4]{ColdingNaber}. Actually, \cite[Theorem 1.4]{ColdingNaber} constructs an $N$-dimensional manifold with nonnegative Ricci and $\AVR>0$ having distinct asymptotic cones containing exactly $k$ lines for any $k=0,\ldots,N-2$. On the variational side, recalling \autoref{cor:IsoperimetricaSuConi}, if a cone splits a line then balls are not strictly stable isoperimetric sets.

Let now $X$ be as in \autoref{def:AsymptoticCone}, and assume for simplicity that $X$ does not contain lines. It is not difficult to show that, roughly speaking, no asymptotic cone of $X$ contains a line if and only if balls centered at tips of asymptotic cones of $X$ are uniformly strictly stable, i.e., the Jacobi operator on the cross sections of such cones is uniformly positive definite. Moreover, in the notation of \autoref{thm:ABFPexistence}, it is possible to prove that if no asymptotic cone of $X'$ contains a line then the hypothesis of \autoref{thm:ABFPexistence} holds true \cite[Lemma 4.2]{AntBruFogPoz}.\\
This yields another, seemingly more intrinsic, existence result: if $X=\R^k\times X'$ is $\RCD(0,N)$ with $\AVR(X)>0$ and $X'$ does not contain lines, if no asymptotic cone of $X'$ contains a line (equivalently, isoperimetric sets in asymptotic cones to $X'$ are uniformly strictly stable), then isoperimetric sets on $X$ exist for large volumes.

\medskip

The previous geometric discussion has a significant application to the class of Riemannian manifolds with nonnegative sectional curvature. Indeed, Toponogov Theorem \cite[Theorem 12.2.2]{Petersen2016} implies that any such manifold has a unique asymptotic cone, which contains lines if and only if the manifold does, see \cite{Kasue88, GuijarroKapovitch95, ShiohamaBOOK, MashikoNaganoOtsuka} or \cite[Theorem 4.6]{AntBruFogPoz} for a proof in the $\AVR>0$ case. Therefore assumptions in \autoref{thm:ABFPexistence} are automatically matched in this class and we deduce the following clean existence result.

\begin{theorem}[{\cite[Theorem 1.3]{AntBruFogPoz}}]\label{thm:ExistenceSectional}
Let $N \in \N$ with $N\ge 2$. Let $(M,g)$ be a complete $N$-dimensional Riemannian manifold with nonnegative sectional curvature and $\AVR(M)>0$\footnote{More generally, the theorem holds for an $N$-dimensional Alexandrov space with curvature bounded below by zero, that we do not introduce here, and positive $\AVR$, see \cite{APPSb}.}.\\
Then there exists $V_0$ such that for any $V\ge V_0$ there exists an isoperimetric region of volume $V$ on $X$.
\end{theorem}

We mention that the somehow dual setting with respect to the one in \autoref{thm:ExistenceSectional} - namely, the one with minimal volume growth assumption - was recently considered in \cite{AntonelliPozzettaAlexandrov}, obtaining the same result of existence for large volumes.

Anyway, the following general question remains open.

\begin{question}\label{question:Existence0N}
Let $N \in \N$ with $N\ge 3$\footnote{The case $N=2$ has a positive answer by \cite{AntonelliPozzettaAlexandrov}, after \cite{RitoreExistenceSurfaces01}.}. Let $(X,\dist,\haus^N)$ be a noncompact $\RCD(0,N)$ space such that $\inf_{x \in X} \haus^N(B_1(x))>0$. Do isoperimetric sets exist for any volume? Up to the author's knowledge, without additional assumptions, the question is also open in the class of manifolds with nonnegative sectional curvature and even in the one of convex bodies in the Euclidean space.
\end{question}

We stress that, as nonexistence of isoperimetric sets is basically equivalent to the drifting towards infinity of minimizing sequences for the problem, \autoref{question:Existence0N} asks for a better understanding the structure at infinity of manifolds with nonnegative curvature.

\medskip

Exploiting the sharp differential inequalities of the profile, the proof of \autoref{thm:ABFPexistence} extremely simplifies as follows. The next straightforward argument was also observed by G. Antonelli, E. Bruè and M. Fogagnolo.

\begin{proof}[Proof of \autoref{thm:ABFPexistence}]
Assume for simplicity that $k=0$. Recalling that \eqref{eq:DisugIsoperimetricaAVR} is sharp for large volumes, we can choose $V_0>0$ such that $I_X(V)\le N((\AVR(X)+\eps/2)\omega_N)^{\frac1N}V^{\frac{N-1}{N}}$ for any $V\ge V_0$. Fix $V\ge V_0$ and assume by contradiction that there does not exist isoperimetric regions of volume $V$ on $X$. Apply the asymptotic mass decomposition \autoref{thm:AsymptoticMassDecomposition} for such volume $V$. By \autoref{rem:UnicoPezzo}, there is a set $E \subset Y$ with $P(E)=I_X(\haus^N(E))$ and measure $V$ in a pmGH limit at infinity $(Y,\dist_Y,\haus^N)$ of $X$. But then \autoref{thm:DisugisoperimetricaAVR} and the assumptions imply
\[
N((\AVR(X)+\eps)\omega_N)^{\frac1N}V^{\frac{N-1}{N}} \le I_Y(V) \le P(E) =I_X(V) \le
N((\AVR(X)+\eps/2)\omega_N)^{\frac1N}V^{\frac{N-1}{N}},
\]
yielding a contradiction.
\end{proof}

Apart from the problem of existence like the one in \autoref{question:Existence0N}, we mention that the characterization of \emph{qualitative} properties of isoperimetric sets still is an open problem, at least in the generality considered in this section. For instance, we record the following question.

\begin{question}
    Let $N \in \N$ with $N\ge 2$. Let $(M,g)$ be a noncompact $N$-dimensional manifold with nonnegative sectional curvature such that $\inf_{x \in M} \haus^N(B_1(x))>0$. Are isoperimetric sets of sufficiently large volume convex? What if $\AVR(M)>0$?
\end{question}

Related to the recent \cite{NiuNonsmoothIsoperimetrico}, we mention also the following

\begin{question}
    Let $N \in \N$ with $N\ge 8$. Is it possible to construct an example of a noncompact $N$-dimensional manifold $(M,g)$ with $\Ric\ge 0$, $\AVR(M)>0$ and such that there exist isoperimetric sets of arbitrarily large volumes with nonempty singular set, i.e., such that there exist points in the topological boundary such that the blowup of the set at such points is a not a halfspace?
\end{question}

\subsection{L\'{e}vy--Gromov isoperimetric inequality}\label{sec:LevyGromov}

As a final application of the differential inequalities of the isoperimetric profile, we present a proof of the classical L\'{e}vy--Gromov isoperimetric inequality in the noncollapsed $\RCD$ setting. The fundamental ideas in the proof we include are due to \cite[Théorème 2.4.3]{Bayle03}. We provide a little simplification concerning the rigidity part, which was, in turn, based on the rigidity in \cite[Théorème 6.14]{Gallotast}. The L\'{e}vy--Gromov inequality was firstly proved in the $\RCD$ setting in \cite{CavallettiMondinoIsopIneq17}.

\begin{theorem}[L\'{e}vy--Gromov isoperimetric inequality]\label{thm:LevyGromov}
	Let $N \in \N$ with $N\ge 2$. Let $(X,\dist,\haus^N)$ be an $\RCD(N-1,N)$ space. Then
	\[
	\frac{I_X(t\haus^N(X))}{\haus^N(X)} \ge \frac{I_{\S^N}(t\haus^N(\S^N))}{\haus^N(\S^N)},
	\]
	for any $t\in[0,1]$. Moreover equality holds for some $t \in (0,1)$ if and only if $X$ is isometric to the spherical suspension over an $\RCD(N-2,N-1)$ space. In particular, if $X$ is a Riemannian manifold, equality holds for some $t \in (0,1)$ if and only if $X$ is isometric to $\S^N$.
\end{theorem}

The original proof of the inequality on Riemannian manifolds is contained in \cite{GromovLevyGromovOriginale}, see also \cite[Appendix I]{MilmanSchechtmanGromov} and \cite[Appendix C]{Gromovmetric}. The L\'{e}vy--Gromov inequality has been greatly generalized in \cite{Milmanmodels} to weighted manifolds with a possibly negative lower bound on the generalized Ricci curvature and an upper bound on the diameter. The previous result was then extended to the $\RCD$ setting in \cite{CavallettiMondinoIsopIneq17}. Quantitative versions of the L\'{e}vy--Gromov inequality have been obtained in \cite{BerardBessonGallotiso}, which was also recovered in \cite[Lemma 3.1]{CavallettiMondinoSemolaQuantitativeObata}, and in \cite{CavallettiMaggiMondino} at the level of $\CD$ spaces.

\begin{proof}[Proof of \autoref{thm:LevyGromov}]
Let
\[
f_1(t) \eqdef \big(I_{\S^N}(t\haus^N(\S^N))/\haus^N(\S^N)\big)^{\frac{N}{N-1}}, \qquad f_2(t) \eqdef \big(I_X(t\haus^N(X))/\haus^N(X)\big)^{\frac{N}{N-1}}.
\]
By \autoref{thm:DifferentialInequalitiesProfile}, one readily checks that $\overline{D}^2 f_2(t) \le -N f_2^{\frac{2-N}{N}}$ for any $t \in (0,1)$, while $f_1''(t)=-N f_1^{\frac{2-N}{N}}$ on $(0,1)$. So we can apply \autoref{lem:Comparison} with $g(y)=-Ny^{\frac{2-N}{N}}$ to get that $f_2\ge f_1$ on $[0,1]$, which proves the L\'{e}vy--Gromov inequality.

Now suppose there is $t_0\in(0,1/2]$ such that $f_1(t_0)=f_2(t_0)$. As in \cite[Théorème 2.4.3]{Bayle03}, if $t_0<1/2$, we want to show that $f_1(1/2)=f_2(1/2)$ by applying \autoref{lem:Comparison} to $f_2$ and suitably defined comparison functions. Let $v \in (0,\haus^N(\S^N)/2)$ and define the auxiliary function $\varphi_{v}:[0,1/2]\to \R$ by
\[
\varphi_{v}(t)\eqdef \frac{I_{\S^N}(2vt)}{2v}.
\]
The functions $\varphi_v$ are equivalent to the model functions introduced in \cite[p. 72]{Gallotast}. Let also denote $f_{v}\eqdef \varphi_{v}^{\frac{N}{N-1}}$. By strict concavity of $I_{\S^N}$ we have
\[
\varphi_v(t) = t \frac{I_{\S^N}(2vt)}{2vt} > t \frac{I_{\S^N}(t\haus^N(\S^N))}{t\haus^N(\S^N)} = \frac{I_{\S^N}(t\haus^N(\S^N))}{\haus^N(\S^N)},
\]
on $(0,1/2]$, and $\varphi_v'(t) = I_{\S^N}'(2vt)>I_{\S^N}'(t\haus^N(\S^N))$ on $(0,1/2]$. Hence
\begin{equation*}
    f_v(t) > f_1(t) ,
    \qquad
    f_v'(t) > f_1'(t),
\end{equation*}
on $(0,1/2]$. By \autoref{lem:Comparison}, since $f_1(t_0)=f_2(t_0)$, then there exists the derivative $f_2'(t_0)= f_1'(t_0)$. Assuming $t_0 \in (0,1/2)$, let $f_3:[t_0,1/2]\to \R$ be defined by
\[
f_3(t) \eqdef f_v(t) - f_v(t_0) + f_2(t_0).
\]
It is immediate to check that $f_3''(t)=-N f_v(t)^{\frac{2-N}{N}} \ge -N f_3(t)^{\frac{2-N}{N}}$ for $t \in (t_0,1/2)$. Moreover $f_3(t_0)=f_2(t_0)$ and $f_3'(t_0) >f_1'(t_0)=f_2'(t_0)$. Hence there are $t_i \to t_0^+$ such that $f_3(t_i)>f_2(t_i)$. Hence item (1) in \autoref{lem:Comparison} yields that $f_3(t) > f_2(t)$ for any $t \in(t_0,1/2]$. Letting $v\to \haus^N(\S^N)/2$ we deduce
\[
f_1(t) \le f_2(t) \le \lim_{v\to \haus^N(\S^N)/2}  f_v(t) - f_v(t_0) + f_2(t_0) = f_1(t)  - f_1(t_0) + f_2(t_0)  = f_1(t),
\]
for any $t \in (t_0,1/2]$. In particular $f_1(1/2)=f_2(1/2)$. Hence \autoref{lem:Comparison} implies that there exists the derivative $f_2'(1/2)=f_1'(1/2)$, which implies $\exists\, I_X'(\haus^N(X)/2)=0$. Let $E\subset X$ be an isoperimetric set of volume $\haus^N(E)=\haus^N(X)/2$, which exists by direct method and compactness of $X$, see \autoref{thm:BonnetMyers}. By \autoref{rem:DerivataProfiloBarriera}, $E$ has mean curvature barrier $c=0$.\\
Now the claimed rigidity follows by arguments analogous to ones in the proof of \cite[Théorème 6.14]{Gallotast}, with a slight simplification which exploits the vanishing of the mean curvature barrier of $E$. On an $\RCD(N-1,N)$ space, for an isoperimetric set $E$ with zero mean curvature barrier, analogous Heintze--Karcher-type estimates as the ones in \autoref{cor:HK}, again obtained by integration and Gauss--Green \cite[Theorem 1.6]{BPSGaussGreen}, read
\[
|\haus^N(E_r)- \haus^N(E) | \le P(E) \int_0^{|r|} (\cos s)_+^{N-1} \de s,
\]
where $E_r\eqdef \{ x \st \dist_E^s(x)<r\}$, for $\dist_E^s$ as in \autoref{def:SignedDistance}, and $(\cdot)_+$ denotes positive part, for $r \in\R$. Denoting $D\eqdef \diam(X)$ and $d\eqdef \sup_{x \in E} \dist(x,X\setminus E)$, we have that $\sup_{y \in X\setminus E} \dist(y, E) \le D-d$; indeed, if $y\in X\setminus E$ satisfies $\dist(y,E)=\sup_{X\setminus E} \dist(\cdot, E)$, letting $x \in E$ such that $\dist(x,X\setminus E)= \sup_{E} \dist(\cdot, X\setminus E)$, we can take a geodesic $\gamma$ from $x$ to $y$, which thus intersects $\partial E$ at some point $z$, so that
\[
D \ge \dist(x,y) = \dist(x,z) + \dist(z,y) \ge \dist(x,X\setminus E) + \dist (y,E)= d + \sup_{X\setminus E} \dist(\cdot, E).
\]
Therefore we can estimate
\[
\begin{split}
    \frac{\haus^N(X)}{2} &= \haus^N(E) \le P(E) \int_0^d  (\cos s)_+^{N-1} \de s \eqqcolon P(E)\, F(d),\\
    \frac{\haus^N(X)}{2} &= \haus^N(X\setminus E) \le P(E) \int_0^{D-d}  (\cos s)_+^{N-1} \de s \eqqcolon P(E)\,G(d).
\end{split}
\]
Hence $\frac{P(E)}{\haus^N(E)} \ge \left(\min\{F(d),G(d)  \} \right)^{-1}$. Since $d\mapsto F(d)$ is nondecreasing and $d\mapsto G(d)$ is nonincreasing, then $\min\{F(d),G(d)  \} \le F(\overline{d})$ where $\overline{d}\in(0,D)$ satisfies $F(\overline{d})=G(\overline{d})$. If $\overline{d}\le D/2$, then $F(\overline{d})\le F(D/2)$. If $\overline{d}>D/2$, then $G(\overline{d}) \le G(D/2)=F(D/2)$. Hence in any case  $\min\{F(d),G(d)  \} \le F(\overline{d}) \le \int_0^{D/2}  (\cos s)_+^{N-1} \de s$. Recalling that $f_1(1/2)=f_2(1/2)$, we have
\[
\begin{split}
\left( \int_0^{\frac\pi2}  (\cos s)_+^{N-1} \de s \right)^{-1} &=
\frac{I_{\S^N}(\haus^N(\S^N)/2)}{\haus^N(\S^N)/2}=\frac{I_X(\haus^N(X)/2)}{\haus^N(X)/2}=\frac{P(E)}{\haus^N(E)} \\
&\ge \left(\min\{F(d),G(d)  \} \right)^{-1} \ge \left( \int_0^{\frac{D}{2}}  (\cos s)_+^{N-1} \de s \right)^{-1}.
\end{split}
\]
Recalling the Bonnet--Myers \autoref{thm:BonnetMyers}, we deduce that $D=\pi$, and rigidity follows.
\end{proof}

\appendix

\section{Concave functions, second order differential inequalities and comparison}\label{sec:Appendix}

For a function $f$ defined in a neighborhood of some $x\in \R$ we set
\[
\begin{split}
\overline{D}^2f(x) & \eqdef \limsup_{h\to 0^+} \frac{f(x+h)+f(x-h)-2f(x)}{h^2},\\
\underline{D}^2f(x) &\eqdef \liminf_{h\to 0^+} \frac{f(x+h)+f(x-h)-2f(x)}{h^2},
\end{split}
\]
and
\[
\d_h^2f(x) \eqdef \frac{f(x+h)+f(x-h)-2f(x)}{h^2},
\]
for $h>0$, if $x+h,x-h$ lie in the domain of $f$.
Below we denote by ${\rm int}\, I$ the interior of a set $I\subset \R$. 

\begin{lemma}\label{lem:Concavity}
	Let $I\subset \R$ be an interval and let $f:I\to \R$ be a continuous function. Then the following are equivalent.
	\begin{enumerate}
		\item $f$ is concave.
		\item $\d_h^2f(x)\le 0$ for any $h>0$ and $x \in I$ such that $x+h, x-h \in I$.
		\item $\overline{D}^2f(x)\le0$ for any $x \in {\rm int}\,I$.
		\item $\underline{D}^2f(x)\le0$ for any $x \in {\rm int}\,I$.
	\end{enumerate}
\end{lemma}

\begin{proof}
Implications (2)$\Rightarrow$(3) and (3)$\Rightarrow$(4) are obvious. For the implication (1)$\Rightarrow$(2): by concavity we have $f(x) = f((x+h)/2 + (x-h)/2) \ge (f(x+h)+f(x-h))/2$. Let us now prove that (4)$\Rightarrow$(1). Suppose by contradiction that there exist $a<b$ in $I$, $\eta>0$ and $\lambda \in (0,1)$ such that $f((1-\lambda)a+\lambda b) < (1-\lambda)f(a) + \lambda f(b) - \eta$. Hence there is $\eps>0$ such that $f_\eps(x) \eqdef f(x) -\eps x^2$ satisfies $f_\eps((1-\lambda)a+\lambda b) < (1-\lambda)f_\eps(a) + \lambda f_\eps(b)$. If $\ell:\R\to\R$ is the function parametrizing the line through $(a,f_\eps(a))$ and $(b,f_\eps(b))$, the function $g_\eps\eqdef f_\eps - \ell:[a,b]\to\R$ has a minimum at some $x_0$ in the interior $(a,b)$. Hence
\[
-2\eps \overset{(4)}{\ge} \underline{D}^2 g_\eps (x_0) =\liminf_{h\to0^+} \frac{1}{h^2}( g_\eps (x_0+h) + g_\eps (x_0-h) - 2 g_\eps (x_0)) \ge 0,
\]
by minimality at $x_0$, which gives a contradiction.
\end{proof}

The next corollary was observed also in \cite[Sect. B.3.1]{Bayle03}.

\begin{corollary}\label{cor:EquivDoverDunder}
	Let $I\subset \R$ be an open interval and let $f,F:I\to \R$ be continuous functions. Then the following are equivalent.
	\begin{enumerate}
		\item $\overline{D}^2f(x) \le F(x)$ for any $x \in I$.
		\item $\underline{D}^2f(x) \le F(x)$ for any $x \in I$.
		\item $\d_h^2f(x)\le \sup_{y\in [x-h,x+h]} F(y)$ for any $x\in I$ and $h>0$ such that $x+h,x-h \in I$.
	\end{enumerate}
\end{corollary}

\begin{proof}
Implications (1)$\Rightarrow$(2) and (3)$\Rightarrow$(1) are obvious. Let us prove (2)$\Rightarrow$(3). Fix $x_0\in I$ and let $0<h<h'$ such that $(x_0-h',x_0+h')\Subset I$. Let $S\eqdef\sup_{[x_0-h',x_0+h']} F$ and $f_{h'}(x)\eqdef f(x) - Sx^2/2$. Hence $\underline{D}^2 f_{h'}(x)\le F(x) - S \le 0$ for any $x \in [x_0-h',x_0+h']$. By \autoref{lem:Concavity} we deduce that
\[
\d^2_h f(x_0) - \sup_{[x_0-h',x_0+h']} F =\d^2_h f_{h'}(x_0) \le 0.
\]
Letting $h'\searrow h$, (3) follows.
\end{proof}

The next proposition recalls that viscosity, distributional and pointwise formulation of differential inequalities like the one satisfied by the isoperimetric profile are all equivalent. We include a short proof for the convenience of the reader.
%{\color{red} Mi pare che non ci sia sui vari Bayle e Bayle-Rosales}

\begin{proposition}\label{prop:EquivalenceDifferentialInequalities}
	Let $I\subset \R$ be an open interval and let $f:I\to \R$, $g:{\rm Im}\,(f)\to\R$ be continuous functions. Suppose that $f$ is bounded from below. Then the following are equivalent.
	\begin{enumerate}
		\item $f''\le g(f)$ in the viscosity sense on $I$, i.e., for any $x \in I$ and any smooth function $\varphi$ defined in a neighborhood of $x$ such that $\varphi - f$ has a local maximum at $x$, there holds $\varphi''(x)\le g(f(x))$.
		%for any $x \in I$ and any smooth function $\varphi$ defined in a neighborhood of $x$ such that $\varphi \le f$ and $\varphi(x)=f(x)$, there holds $\varphi''(x)\le g(f(x))$.
		\item $f''\le g(f)$ in the sense of distributions on $I$, i.e.,
		\[
		\int f\varphi'' \de x \le \int g(f) \, \varphi \de x,
		\]
		for every $\varphi \in C^\infty_c(I)$ with $\varphi\ge 0$.
		\item $\overline{D}^2f(x) \le g(f(x))$ for every $x \in I$.
	\end{enumerate}
\end{proposition}

\begin{proof}
Without loss of generality, we assume that $f$ is nonnegative.
\begin{itemize}
    \item[(1)$\Rightarrow$(2)]
    Let $\varphi\in C^\infty_c(a,b)$ for some $a,b \in I$ with $\varphi\ge 0$. Let $M\eqdef \sup_{[a,b]} f$. For $\eps>0$ and $x \in [a,b]$ we define the inf-convolution $f_\eps(x) \eqdef \inf_{y \in [a,b]} \{f(y) + |x-y|^2/\eps \}$. Observe that $f_\eps\le f$, $f_\eps$ is semiconcave, and $f_\eps\nearrow f$ pointwise on $[a,b]$, hence uniformly by Dini's monotone convergence. Let $y_x\in[a,b]$ be such that $f_\eps(x)=f(y_x)+|x-y_x|^2/\eps$, hence $|x-y_x|\le \sqrt{\eps M}$ since $f$ is nonnegative.\\
    For $\eps$ small, fix $x_0\in(a+\sqrt{\eps M},b-\sqrt{\eps M})$ such that $f_\eps$ is twice differentiable. Define $\psi(x)\eqdef f_\eps(x_0)+ f_\eps'(x_0)(x-x_0) + (f_\eps''(x_0)-\eta)(x-x_0)^2/2$, for $\eta>0$. Hence $\psi\le f_\eps$ in a neighborhood of $x_0$ and $\psi(x_0)=f_\eps(x_0)$. Let $\widetilde
    \psi(x)\eqdef \psi(x+x_0-y_{x_0})$. It is readily checked that $\widetilde\psi-f$ has a local maximum at $y_{x_0} \in (a,b)$. Hence (1) implies
    \[
    f_\eps''(x_0)-\eta = \widetilde\psi''(y_{x_0})  \le g(f(y_{x_0})),
    \]
    and letting $\eta\to0$ we get
    \[
    f_\eps''(x_0) \le \sup \left\{ g(f) \st [x_0-\sqrt{\eps M}, x_0+ \sqrt{\eps M}] \right\}
    \qquad \forall\,x_0\in(a+\sqrt{\eps M},b-\sqrt{\eps M}).
    \]
    For $\eps$ small enough we have ${\rm spt}\, \varphi \subset (a+2\sqrt{\eps M}, b - 2\sqrt{\eps M})$. Since $\overline{D}^2f_\eps$ is uniformly bounded above, we can apply Fatou's Lemma to get
    \[
    \begin{split}
        \int f_\eps \varphi'' 
        &= \lim_{h\to 0^+} \int_a^b f_\eps \, \d^2_h\varphi
        =  \lim_{h\to 0^+} \int_a^b \d^2_h f_\eps \, \varphi
        \le \int_a^b \overline{D}^2f_\eps \, \varphi
        \le \int_a^b \varphi (x)\,\,\sup_{[x-\sqrt{\eps M}, x+ \sqrt{\eps M}]} g(f(y))  \de x
    \end{split}
    \]
    where we used that $\overline{D}^2f_\eps= f_\eps''$ almost everywhere. Letting $\eps\to0$, (2) follows.
    
    \item[(2)$\Rightarrow$(3)]
    Let $\rho_\eps$ be a nonnegative symmetric mollifier for $\eps\in(0,1)$ with ${\rm spt}\, \rho_\eps\Subset (-\eps,\eps)$. Let $f_\eps\eqdef f \star \rho_\eps$, $(g(f))_\eps\eqdef (g\circ f)\star \rho_\eps$, and fix $x_0 \in I$ such that $(x_0-2\eps_0,x_0+2\eps_0)\subset I$ for $\eps_0\in(0,1)$ small enough. Let $I_\eps\eqdef (x_0-\eps,x_0+\eps)$. Since $\rho_\eps\ge0$, (2) implies that $f_\eps'' \le (g(f))_\eps$ pointwise on $I_{\eps_0}$ for any $\eps<\eps_0$. Then \autoref{cor:EquivDoverDunder} implies that
    \[
    \d^2_h f_\eps (x_0)\le \sup_{[x_0-h,x_0+h]} (g(f))_\eps(y) \qquad\qquad\forall h\in(0,\eps_0/2).
    \]
    Letting first $\eps\to0^+$ we deduce
    \[
    \d^2_h f (x_0)\le \sup_{[x_0-h,x_0+h]} (g(f(y))) \qquad\qquad\forall h\in(0,\eps_0/2),
    \]
    then taking $\limsup_{h\to0^+}$, (3) follows.
    
    \item[(3)$\Rightarrow$(1)]
    Let $x, \varphi$ be as in (1). Without loss of generality we can assume that $\varphi\le f$ and $\varphi(x)=f(x)$. Hence
    \[
    \varphi''(x)= \limsup_{h\to0^+} \d^2_h \varphi(x) \le \limsup_{h\to0^+} \d^2_h f(x) \overset{(3)}{\le} g(f(x)).
    \]
\end{itemize}
\end{proof}

We conclude with an elementary comparison result, whose content is analogous to \cite[Lemme C.2.1]{Bayle03}.

\begin{proposition}\label{lem:Comparison}
	Let $a,b\in\R$. Let $f_1,f_2:[a,b]\to [0,+\infty)$, $g:(0,+\infty)\to\R$ be continuous functions. Suppose that $f_i|_{(a,b)}>0$ for $i=1,2$, $g$ is nondecreasing, and that
	\[
	\overline{D}^2f_1(x)\ge g(f_1(x)), \qquad
	\underline{D}^2f_2(x)\le g(f_2(x)),
	\]
	for any $x \in (a,b)$. Then the following holds.
	\begin{enumerate}
		\item If $f_1(a)=f_2(a)$ and there is $a'\in(a,b]$ such that $f_1(a')>f_2(a')$, then $f_1(x)>f_2(x)$ for any $x \in [a',b]$.
		
		\item If $f_1(b)=f_2(b)$ and there is $b'\in[a,b)$ such that $f_1(b')>f_2(b')$, then $f_1(x)>f_2(x)$ for any $x \in [a,b']$.
		
		\item If $f_1(a)=f_2(a)$ and $f_1(b)=f_2(b)$, then $f_2(x)\ge f_1(x)$ for any $x \in [a,b]$. Moreover, if $f_2(x_0)=f_1(x_0)$ for some $x_0 \in (a,b)$, then
		\begin{equation*}
		\frac{\d^+}{\d x} f_1(x_0) = \frac{\d^+}{\d x} f_2(x_0) , 
		\qquad
		\frac{\d^-}{\d x} f_1(x_0) = \frac{\d^-}{\d x} f_2(x_0),
		\end{equation*}
		where $\tfrac{\d^\pm}{\d x} f_i(x_0)$ denotes right or left derivative of $f_i$ at $x_0$\footnote{Observe that the latter exist as $f_1$ (resp. $f_2$) is locally semiconvex (resp. semiconcave) on $(a,b)$ by \autoref{lem:Concavity}, \autoref{cor:EquivDoverDunder} and continuity of $g$.}.
	\end{enumerate}
\end{proposition}

\begin{proof}
We prove the items separately.
\begin{enumerate}
    \item Let $x_1\eqdef \inf\{ x \in[a,a'] \st f_1(t)>f_2(t) \,\,\forall\,t\in(x,a']\}$ and $x_2\eqdef \sup\{ x \in[a',b] \st f_1(t)>f_2(t) \,\,\forall\,t\in[a',x)\}$. Since $f_1(a)=f_2(a)$ and $f_1, f_2$ are continuous, then $f_1(x_1)=f_2(x_1)$. Since $g$ is nondecreasing, the function $F\eqdef f_2-f_1$ satisfies $\underline{D}^2F(x) \le 0$ on $(x_1,x_2)$, hence $F$ is concave on $[x_1,x_2]$ by \autoref{lem:Concavity}. Since $F\le 0$  on $[x_1,x_2]$, $F(x_1)=0$ and $F(a')<0$, then $F(x_2)<0$. By definition of $x_2$ as supremum, this implies $x_2=b$.
    
    \item Analogous to item (1).
    
    \item If $f_2(a')<f_1(a')$ for some $a'\in(a,b)$, then (1) and (2) imply that $f_2(a)<f_1(a)$ or $f_2(b)<f_1(b)$, against the assumptions. Suppose now that $f_2(x_0)=f_1(x_0)$ for some $x_0 \in (a,b)$. Let $F\eqdef f_2-f_1$. The function $F$ is nonnegative and has a minimum at $x_0$. By continuity, there exist $C,h>0$ such that $\underline{D}^2 F(x) \le g(f_2(x))- g(f_1(x)) \le C$ for any $x \in (x_0-h,x_0+h) \Subset (a,b)$. \autoref{cor:EquivDoverDunder} and \autoref{lem:Concavity} imply that $G(x)\eqdef F(x)-C x^2$ is concave on $(x_0-h,x_0+h)$. Hence
    \[
    \frac{\d^-}{\d x} F (x_0) -2C x_0 = \frac{\d^-}{\d x} G (x_0) \ge \frac{\d^+}{\d x} G (x_0) = \frac{\d^+}{\d x} F (x_0) -2C x_0.
    \]
    Since $F$ has a minimum at $x_0$, the previous estimate implies $ 0 \ge \tfrac{\d^-}{\d x} F (x_0) \ge  \tfrac{\d^+}{\d x} F (x_0) \ge 0$. Hence
    \[
    0 = \frac{\d^-}{\d x} F (x_0)
    = \frac{\d^-}{\d x}f_2(x_0) - \frac{\d^-}{\d x}f_1(x_0) = \frac{\d^+}{\d x}f_2(x_0) - \frac{\d^+}{\d x}f_1(x_0) =  \frac{\d^+}{\d x} F (x_0) = 0.
    \]
\end{enumerate}
\end{proof}

\printbibliography[title={References}]

\typeout{get arXiv to do 4 passes: Label(s) may have changed. Rerun} %Questo comando pare che sia per far funzionare poi la compilazione su arXiv

\end{document}